\documentclass{article}
\usepackage{fullpage}

\usepackage{hyperref}

\usepackage{amsmath,amsthm,amssymb}
\usepackage{xspace,enumerate,color,epsfig}
\usepackage{graphicx}
\graphicspath{{.}{./figures/}}
\usepackage{marvosym}

\usepackage{tikzfig}
\usepackage{stmaryrd}
\usepackage{docmute}
\usepackage{keycommand}

\newkeycommand{\pointmap}[style=point][1]{\,\tikz{\node[style=\commandkey{style}] (x) at (0,-0.3) {$#1$};\node [style=none] (3) at (0, 1.0) {};\node [style=none] (3) at (0, -1.0) {};\draw (x)--(0,0.7);}\,}

\newkeycommand{\dpoint}[style=point][1]{\,\tikz{\node[style=\commandkey{style},doubled] (x) at (0,-0.3) {$#1$};\draw[boldedge] (x)--(0,0.7);}\,}

\newkeycommand{\typedkpoint}[style=kpoint,edgestyle=][2]{%
\,\begin{tikzpicture}
  \begin{pgfonlayer}{nodelayer}
    \node [style=none] (0) at (0, 1) {};
    \node [style=\commandkey{style}] (1) at (0, -0.75) {$#2$};
    \node [style=right label] (2) at (0.25, 0.5) {$#1$};
  \end{pgfonlayer}
  \begin{pgfonlayer}{edgelayer}
    \draw [\commandkey{edgestyle}] (1) to (0.center);
  \end{pgfonlayer}
\end{tikzpicture}\,}

\newkeycommand{\typedkpointdag}[style=kpoint adjoint,edgestyle=][2]{%
\,\begin{tikzpicture}
  \begin{pgfonlayer}{nodelayer}
    \node [style=none] (0) at (0, -1) {};
    \node [style=\commandkey{style}] (1) at (0, 0.75) {$#2$};
    \node [style=right label] (2) at (0.25, -0.5) {$#1$};
  \end{pgfonlayer}
  \begin{pgfonlayer}{edgelayer}
    \draw [\commandkey{edgestyle}] (0.center) to (1);
  \end{pgfonlayer}
\end{tikzpicture}\,}

\newkeycommand{\bistate}[style=kpoint][1]{\,\begin{tikzpicture}
  \begin{pgfonlayer}{nodelayer}
    \node [style=\commandkey{style}, minimum width=1 cm, inner sep=2pt] (0) at (0, -0.25) {$#1$};
    \node [style=none] (1) at (-0.75, 0) {};
    \node [style=none] (2) at (0.75, 0) {};
    \node [style=none] (3) at (-0.75, 0.75) {};
    \node [style=none] (4) at (0.75, 0.75) {};
  \end{pgfonlayer}
  \begin{pgfonlayer}{edgelayer}
    \draw (1.center) to (3.center);
    \draw (2.center) to (4.center);
  \end{pgfonlayer}
\end{tikzpicture}\,}

\newkeycommand{\dbistate}[style=dkpoint][1]{\,\begin{tikzpicture}
  \begin{pgfonlayer}{nodelayer}
    \node [style=\commandkey{style}, minimum width=1 cm, inner sep=2pt] (0) at (0, -0.25) {$#1$};
    \node [style=none] (1) at (-0.75, 0) {};
    \node [style=none] (2) at (0.75, 0) {};
    \node [style=none] (3) at (-0.75, 1.0) {};
    \node [style=none] (4) at (0.75, 1.0) {};
  \end{pgfonlayer}
  \begin{pgfonlayer}{edgelayer}
    \draw [boldedge] (1.center) to (3.center);
    \draw [boldedge] (2.center) to (4.center);
  \end{pgfonlayer}
\end{tikzpicture}\,}

\newcommand{\bistateadj}[1]{\,\begin{tikzpicture}[yshift=-3mm]
  \begin{pgfonlayer}{nodelayer}
    \node [style=kpointadj, minimum width=1 cm, inner sep=2pt] (0) at (0, 0.5) {$\psi$};
    \node [style=none] (1) at (-0.75, 0.25) {};
    \node [style=none] (2) at (0.75, 0.25) {};
    \node [style=none] (3) at (-0.75, -0.5) {};
    \node [style=none] (4) at (0.75, -0.5) {};
  \end{pgfonlayer}
  \begin{pgfonlayer}{edgelayer}
    \draw (1.center) to (3.center);
    \draw (2.center) to (4.center);
  \end{pgfonlayer}
\end{tikzpicture}\,}

\newkeycommand{\bistatebraket}[style1=kpoint,style2=kpointadj][2]{\,\begin{tikzpicture}
  \begin{pgfonlayer}{nodelayer}
    \node [style=\commandkey{style1}, minimum width=1 cm, inner sep=2pt] (0) at (0, -0.75) {$#1$};
    \node [style=none] (1) at (-0.75, -0.5) {};
    \node [style=none] (2) at (0.75, -0.5) {};
    \node [style=none] (3) at (-0.75, 0.5) {};
    \node [style=none] (4) at (0.75, 0.5) {};
    \node [style=\commandkey{style2}, minimum width=1 cm, inner sep=2pt] (5) at (0, 0.75) {$#2$};
  \end{pgfonlayer}
  \begin{pgfonlayer}{edgelayer}
    \draw (1.center) to (3.center);
    \draw (2.center) to (4.center);
  \end{pgfonlayer}
\end{tikzpicture}\,}

\newcommand{\trace}{\,\begin{tikzpicture}
  \begin{pgfonlayer}{nodelayer}
    \node [style=none] (0) at (0, -0.75) {};
    \node [style=upground] (1) at (0, 0.75) {};
  \end{pgfonlayer}
  \begin{pgfonlayer}{edgelayer}
    \draw [style=boldedge] (0.center) to (1);
  \end{pgfonlayer}
\end{tikzpicture}\,}

\newcommand\discard\trace

\newkeycommand{\pointketbra}[style1=copoint,style2=point][2]{\,%
\begin{tikzpicture}
  \begin{pgfonlayer}{nodelayer}
    \node [style=none] (0) at (0, -1.5) {};
    \node [style=\commandkey{style1}] (1) at (0, -0.75) {$#1$};
    \node [style=\commandkey{style2}] (2) at (0, 0.75) {$#2$};
    \node [style=none] (3) at (0, 1.5) {};
  \end{pgfonlayer}
  \begin{pgfonlayer}{edgelayer}
    \draw (0.center) to (1);
    \draw (2) to (3.center);
  \end{pgfonlayer}
\end{tikzpicture}\,}

\newkeycommand{\twocopointketbra}[style1=copoint,style2=copoint,style3=point][3]{\,%
\begin{tikzpicture}
  \begin{pgfonlayer}{nodelayer}
    \node [style=none] (0) at (-0.75, -1.5) {};
    \node [style=none] (1) at (0.75, -1.5) {};
    \node [style=\commandkey{style1}] (2) at (-0.75, -0.75) {$#1$};
    \node [style=\commandkey{style2}] (3) at (0.75, -0.75) {$#2$};
    \node [style=\commandkey{style3}] (4) at (0, 0.75) {$#3$};
    \node [style=none] (5) at (0, 1.5) {};
  \end{pgfonlayer}
  \begin{pgfonlayer}{edgelayer}
    \draw (0.center) to (2);
    \draw (1.center) to (3);
    \draw (4) to (5.center);
  \end{pgfonlayer}
\end{tikzpicture}\,}

\newkeycommand{\twopointketbra}[style1=copoint,style2=point,style3=point][3]{\,%
\begin{tikzpicture}
  \begin{pgfonlayer}{nodelayer}
    \node [style=none] (0) at (0, -1.5) {};
    \node [style=none] (1) at (-0.75, 1.5) {};
    \node [style=\commandkey{style1}] (2) at (0, -0.75) {$#1$};
    \node [style=\commandkey{style2}] (3) at (-0.75, 0.75) {$#2$};
    \node [style=\commandkey{style3}] (4) at (0.75, 0.75) {$#3$};
    \node [style=none] (5) at (0.75, 1.5) {};
  \end{pgfonlayer}
  \begin{pgfonlayer}{edgelayer}
    \draw (0.center) to (2);
    \draw (1.center) to (3);
    \draw (4) to (5.center);
  \end{pgfonlayer}
\end{tikzpicture}\,}

\newkeycommand{\pointbraket}[style1=point,style2=copoint,estyle=][2]{\,%
\begin{tikzpicture}
  \begin{pgfonlayer}{nodelayer}
    \node [style=\commandkey{style1}] (0) at (0, -0.5) {$#1$};
    \node [style=\commandkey{style2}] (1) at (0, 0.5) {$#2$};
  \end{pgfonlayer}
  \begin{pgfonlayer}{edgelayer}
    \draw [\commandkey{estyle}] (0) to (1);
  \end{pgfonlayer}
\end{tikzpicture}\,}

\newkeycommand{\kpointbraket}[style1=kpoint,style2=kpoint adjoint,estyle=][2]{\,%
\begin{tikzpicture}
  \begin{pgfonlayer}{nodelayer}
    \node [style=\commandkey{style1}] (0) at (0, -0.75) {$#1$};
    \node [style=\commandkey{style2}] (1) at (0, 0.75) {$#2$};
  \end{pgfonlayer}
  \begin{pgfonlayer}{edgelayer}
    \draw [\commandkey{estyle}] (0) to (1);
  \end{pgfonlayer}
\end{tikzpicture}\,}

\newkeycommand{\kpointmap}[style1=kpoint,style2=map][2]{\,%
\begin{tikzpicture}
  \begin{pgfonlayer}{nodelayer}
    \node [style=\commandkey{style1}] (0) at (0, -1.1) {$#1$};
    \node [style=\commandkey{style2}] (1) at (0, 0.2) {$#2$};
    \node [style=none] (2) at (0, 1.5) {};
  \end{pgfonlayer}
  \begin{pgfonlayer}{edgelayer}
    \draw (0) to (1);
    \draw (1) to (2.center);
  \end{pgfonlayer}
\end{tikzpicture}%
\,}

\newkeycommand{\dkpointmap}[style1=dkpoint,style2=dmap][2]{\,%
\begin{tikzpicture}
  \begin{pgfonlayer}{nodelayer}
    \node [style=\commandkey{style1}] (0) at (0, -1.2) {$#1$};
    \node [style=\commandkey{style2}] (1) at (0, 0.4) {$#2$};
    \node [style=none] (2) at (0, 1.7) {};
  \end{pgfonlayer}
  \begin{pgfonlayer}{edgelayer}
    \draw[style=boldedge] (0) to (1);
    \draw[style=boldedge] (1) to (2.center);
  \end{pgfonlayer}
\end{tikzpicture}%
\,}

\newkeycommand{\sandwichtwo}[style1=point,style2=map,style3=map,style4=copoint][4]{\,%
\begin{tikzpicture}
  \begin{pgfonlayer}{nodelayer}
    \node [style=\commandkey{style2}] (0) at (0, -0.75) {$#2$};
    \node [style=\commandkey{style3}] (1) at (0, 0.75) {$#3$};
    \node [style=\commandkey{style1}] (2) at (0, -2) {$#1$};
    \node [style=\commandkey{style4}] (3) at (0, 2) {$#4$};
  \end{pgfonlayer}
  \begin{pgfonlayer}{edgelayer}
    \draw (2) to (0);
    \draw (0) to (1);
    \draw (1) to (3);
  \end{pgfonlayer}
\end{tikzpicture}\,}

\newkeycommand{\longbraket}[style1=point,style2=copoint][2]{\,%
\begin{tikzpicture}
  \begin{pgfonlayer}{nodelayer}
    \node [style=\commandkey{style1}] (0) at (0, -2) {$#1$};
    \node [style=\commandkey{style2}] (1) at (0, 2) {$#2$};
  \end{pgfonlayer}
  \begin{pgfonlayer}{edgelayer}
    \draw (0) to (1);
  \end{pgfonlayer}
\end{tikzpicture}\,}

\newkeycommand{\onb}[style=point]{\ensuremath{\{\,\tikz{\node[style=\commandkey{style}] (x) at (0,-0.3) {$j$};\draw (x)--(0,0.7);}\,\}}\xspace}

\newkeycommand{\copointmap}[style=copoint][1]{\,\tikz{\node[style=\commandkey{style}] (x) at (0,0.3) {$#1$};\draw (x)--(0,-0.7);}\,}

\newkeycommand{\dphasepoint}[style=white phase ddot][1]{\,\begin{tikzpicture}[yshift=-3mm]
  \begin{pgfonlayer}{nodelayer}
    \node [style=none] (0) at (0, 1.25) {};
    \node [style=\commandkey{style}] (1) at (0, 0) {$#1$};
  \end{pgfonlayer}
  \begin{pgfonlayer}{edgelayer}
    \draw [style=boldedge] (0.center) to (1);
  \end{pgfonlayer}
\end{tikzpicture}\,}

\tikzstyle{cdiag}=[matrix of math nodes, row sep=3em, column sep=3em, text height=1.5ex, text depth=0.25ex,inner sep=0.5em]
\tikzstyle{arrow above}=[transform canvas={yshift=0.5ex}]
\tikzstyle{arrow below}=[transform canvas={yshift=-0.5ex}]

\newcommand{\NWbracket}[1]{%
\draw #1 +(-0.25,0.5) -- +(-0.5,0.5) -- +(-0.5,0.25);}

\tikzstyle{env}=[copoint,regular polygon rotate=0,minimum width=0.2cm, fill=black]

\tikzstyle{probs}=[shape=semicircle,fill=white,draw=black,shape border rotate=180,minimum width=1.2cm]

\tikzstyle{every picture}=[baseline=-0.25em,scale=0.5]
\tikzstyle{dotpic}=[] 
\tikzstyle{diredges}=[every to/.style={diredge}]
\tikzstyle{math matrix}=[matrix of math nodes,left delimiter=(,right delimiter=),inner sep=2pt,column sep=1em,row sep=0.5em,nodes={inner sep=0pt},text height=1.5ex, text depth=0.25ex]

\tikzstyle{inline text}=[text height=1.5ex, text depth=0.25ex,yshift=0.5mm]
\tikzstyle{label}=[font=\footnotesize,text height=1.5ex, text depth=0.25ex,yshift=0.5mm]
\tikzstyle{left label}=[label,anchor=east,xshift=1.5mm]
\tikzstyle{right label}=[label,anchor=west,xshift=-1.5mm]

\tikzstyle{braceedge}=[decorate,decoration={brace,amplitude=2mm,raise=-1mm}]
\tikzstyle{small braceedge}=[decorate,decoration={brace,amplitude=1mm,raise=-1mm}]

\tikzstyle{doubled}=[line width=2pt] 
\tikzstyle{boldedge}=[doubled,shorten <=-0.17mm,shorten >=-0.17mm]

\tikzstyle{boldedgedashed}=[very thick,dashed,shorten <=-0.17mm,shorten >=-0.17mm]
\tikzstyle{vboldedgedashed}=[doubled,dashed,shorten <=-0.17mm,shorten >=-0.17mm]
\tikzstyle{left hook arrow}=[left hook-latex]
\tikzstyle{right hook arrow}=[right hook-latex]
\tikzstyle{sembracket}=[line width=0.5pt,shorten <=-0.07mm,shorten >=-0.07mm]

\tikzstyle{causal edge}=[->,thick,gray]
\tikzstyle{timeline}=[thick,gray, dashed]

\tikzstyle{cedge}=[<->,thick,gray!70!white]

\tikzstyle{empty diagram}=[draw=gray!40!white,dashed,shape=rectangle,minimum width=1cm,minimum height=1cm]
\tikzstyle{empty diagram small}=[draw=gray!50!white,dashed,shape=rectangle,minimum width=0.6cm,minimum height=0.5cm]

\tikzstyle{dot}=[inner sep=0.7mm,minimum width=0pt,minimum height=0pt,draw,shape=circle]
\tikzstyle{ddot}=[inner sep=0.7mm,doubled, minimum width=2.5mm,minimum height=2.5mm,draw,shape=circle]

\tikzstyle{black dot}=[dot,fill=black]
\tikzstyle{white dot}=[dot,fill=white]
\tikzstyle{green dot}=[white dot] 
\tikzstyle{gray dot}=[dot,fill=gray!40!white]
\tikzstyle{red dot}=[gray dot] 

\tikzstyle{black ddot}=[ddot,fill=black]
\tikzstyle{white ddot}=[ddot,fill=white]
\tikzstyle{gray ddot}=[ddot,fill=gray!40!white]

\tikzstyle{small dot}=[inner sep=0.4mm,minimum width=0pt,minimum height=0pt,draw,shape=circle]

\tikzstyle{small black dot}=[small dot,fill=black]
\tikzstyle{small white dot}=[small dot,fill=white]
\tikzstyle{small gray dot}=[small dot,fill=gray!40!white]

\tikzstyle{causal dot}=[inner sep=0.4mm,minimum width=0pt,minimum height=0pt,draw=white,shape=circle,fill=gray!40!white]

\tikzstyle{white phase dot}=[dot,fill=white]
\tikzstyle{white phase ddot}=[ddot,fill=white]

\tikzstyle{gray phase dot}=[dot,fill=gray!40!white]
\tikzstyle{gray phase ddot}=[ddot,fill=gray!40!white]
\tikzstyle{grey phase dot}=[gray phase dot]
\tikzstyle{grey phase ddot}=[gray phase ddot]

\tikzstyle{cnot}=[fill=white,shape=circle,inner sep=-1.4pt]
\tikzstyle{hadamard}=[square box,inner sep=0 pt,font=\tiny\sf,minimum height=3mm,minimum width=3mm]
\tikzstyle{dhadamard}=[hadamard,doubled]
\tikzstyle{antipode}=[white dot,inner sep=0.3mm,font=\footnotesize]

\tikzstyle{scalar}=[diamond,draw,inner sep=0.5pt,font=\small]
\tikzstyle{dscalar}=[diamond,doubled, draw,inner sep=0.5pt,font=\small]

\tikzstyle{small box}=[rectangle,inline text,fill=white,draw,minimum height=5mm,yshift=-0.5mm,minimum width=5mm,font=\small]
\tikzstyle{small gray box}=[small box,fill=gray!30]
\tikzstyle{medium box}=[rectangle,inline text,fill=white,draw,minimum height=5mm,yshift=-0.5mm,minimum width=10mm,font=\small]
\tikzstyle{square box}=[small box] 
\tikzstyle{medium gray box}=[small box,fill=gray!30]
\tikzstyle{large box}=[rectangle,inline text,fill=white,draw,minimum height=5mm,yshift=-0.5mm,minimum width=15mm,font=\small]
\tikzstyle{large gray box}=[small box,fill=gray!30]

\tikzstyle{point}=[regular polygon,regular polygon sides=3,draw,scale=0.75,inner sep=-0.5pt,minimum width=9mm,fill=white,regular polygon rotate=180]
\tikzstyle{copoint}=[regular polygon,regular polygon sides=3,draw,scale=0.75,inner sep=-0.5pt,minimum width=9mm,fill=white]
\tikzstyle{dpoint}=[point,doubled]
\tikzstyle{dcopoint}=[copoint,doubled]

\tikzstyle{tinypoint}=[regular polygon,regular polygon sides=3,draw,scale=0.55,inner sep=-0.15pt,minimum width=6mm,fill=white,regular polygon rotate=180]

\tikzstyle{white point}=[point]
\tikzstyle{green point}=[white point] 
\tikzstyle{white copoint}=[copoint]
\tikzstyle{gray point}=[point,fill=gray!40!white]
\tikzstyle{gray dpoint}=[gray point,doubled]
\tikzstyle{red point}=[gray point] 
\tikzstyle{gray copoint}=[copoint,fill=gray!40!white]
\tikzstyle{gray dcopoint}=[gray copoint,doubled]

\tikzstyle{tiny gray point}=[tinypoint,fill=gray!40!white]

\tikzstyle{diredge}=[->]
\tikzstyle{rdiredge}=[<-]
\tikzstyle{thickdiredge}=[->, very thick]
\tikzstyle{pointer edge}=[->,very thick,gray]
\tikzstyle{dashed edge}=[dashed]
\tikzstyle{thick dashed edge}=[very thick,dashed]
\tikzstyle{thick gray dashed edge}=[thick dashed edge,gray!40]
\tikzstyle{thick map edge}=[very thick,|->]

\makeatletter
\newcommand{\boxshape}[3]{%
\pgfdeclareshape{#1}{
\inheritsavedanchors[from=rectangle] 
\inheritanchorborder[from=rectangle]
\inheritanchor[from=rectangle]{center}
\inheritanchor[from=rectangle]{north}
\inheritanchor[from=rectangle]{south}
\inheritanchor[from=rectangle]{west}
\inheritanchor[from=rectangle]{east}
\backgroundpath{
\southwest \pgf@xa=\pgf@x \pgf@ya=\pgf@y
\northeast \pgf@xb=\pgf@x \pgf@yb=\pgf@y

\@tempdima=#2
\@tempdimb=#3

\pgfpathmoveto{\pgfpoint{\pgf@xa - 5pt + \@tempdima}{\pgf@ya}}
\pgfpathlineto{\pgfpoint{\pgf@xa - 5pt - \@tempdima}{\pgf@yb}}
\pgfpathlineto{\pgfpoint{\pgf@xb + 5pt + \@tempdimb}{\pgf@yb}}
\pgfpathlineto{\pgfpoint{\pgf@xb + 5pt - \@tempdimb}{\pgf@ya}}
\pgfpathlineto{\pgfpoint{\pgf@xa - 5pt + \@tempdima}{\pgf@ya}}
\pgfpathclose
}
}}

\boxshape{NEbox}{0pt}{5pt}
\boxshape{SEbox}{0pt}{-5pt}
\boxshape{NWbox}{5pt}{0pt}
\boxshape{SWbox}{-5pt}{0pt}
\boxshape{EBox}{-3pt}{3pt}
\boxshape{WBox}{3pt}{-3pt}
\makeatother

\tikzstyle{cloud}=[shape=cloud,draw,minimum width=1.5cm,minimum height=1.5cm]

\tikzstyle{map}=[draw,shape=NEbox,inner sep=2pt,minimum height=6mm,fill=white]
\tikzstyle{mapdag}=[draw,shape=SEbox,inner sep=2pt,minimum height=6mm,fill=white]
\tikzstyle{mapadj}=[draw,shape=SEbox,inner sep=2pt,minimum height=6mm,fill=white]
\tikzstyle{maptrans}=[draw,shape=SWbox,inner sep=2pt,minimum height=6mm,fill=white]
\tikzstyle{mapconj}=[draw,shape=NWbox,inner sep=2pt,minimum height=6mm,fill=white]

\tikzstyle{dbox}=[draw,doubled,shape=rectangle,inner sep=2pt,minimum height=6mm,minimum width=6mm,fill=white]
\tikzstyle{dmap}=[draw,doubled,shape=NEbox,inner sep=2pt,minimum height=6mm,fill=white]
\tikzstyle{dmapdag}=[draw,doubled,shape=SEbox,inner sep=2pt,minimum height=6mm,fill=white]
\tikzstyle{dmapadj}=[draw,doubled,shape=SEbox,inner sep=2pt,minimum height=6mm,fill=white]
\tikzstyle{dmaptrans}=[draw,doubled,shape=SWbox,inner sep=2pt,minimum height=6mm,fill=white]
\tikzstyle{dmapconj}=[draw,doubled,shape=NWbox,inner sep=2pt,minimum height=6mm,fill=white]

\tikzstyle{ddmap}=[draw,doubled,dashed,shape=NEbox,inner sep=2pt,minimum height=6mm,fill=white]
\tikzstyle{ddmapdag}=[draw,doubled,dashed,shape=SEbox,inner sep=2pt,minimum height=6mm,fill=white]
\tikzstyle{ddmapadj}=[draw,doubled,dashed,shape=SEbox,inner sep=2pt,minimum height=6mm,fill=white]
\tikzstyle{ddmaptrans}=[draw,doubled,dashed,shape=SWbox,inner sep=2pt,minimum height=6mm,fill=white]
\tikzstyle{ddmapconj}=[draw,doubled,dashed,shape=NWbox,inner sep=2pt,minimum height=6mm,fill=white]

\boxshape{sNEbox}{0pt}{3pt}
\boxshape{sSEbox}{0pt}{-3pt}
\boxshape{sNWbox}{3pt}{0pt}
\boxshape{sSWbox}{-3pt}{0pt}
\tikzstyle{smap}=[draw,shape=sNEbox,fill=white]
\tikzstyle{smapdag}=[draw,shape=sSEbox,fill=white]
\tikzstyle{smapadj}=[draw,shape=sSEbox,fill=white]
\tikzstyle{smaptrans}=[draw,shape=sSWbox,fill=white]
\tikzstyle{smapconj}=[draw,shape=sNWbox,fill=white]

\tikzstyle{dsmap}=[draw,dashed,shape=sNEbox,fill=white]
\tikzstyle{dsmapdag}=[draw,dashed,shape=sSEbox,fill=white]
\tikzstyle{dsmaptrans}=[draw,dashed,shape=sSWbox,fill=white]
\tikzstyle{dsmapconj}=[draw,dashed,shape=sNWbox,fill=white]

\boxshape{mNEbox}{0pt}{10pt}
\boxshape{mSEbox}{0pt}{-10pt}
\boxshape{mNWbox}{10pt}{0pt}
\boxshape{mSWbox}{-10pt}{0pt}
\tikzstyle{mmap}=[draw,shape=mNEbox]
\tikzstyle{mmapdag}=[draw,shape=mSEbox]
\tikzstyle{mmaptrans}=[draw,shape=mSWbox]
\tikzstyle{mmapconj}=[draw,shape=mNWbox]

\tikzstyle{mmapgray}=[draw,fill=gray!40!white,shape=mNEbox]
\tikzstyle{smapgray}=[draw,fill=gray!40!white,shape=sNEbox]

\makeatletter
\pgfdeclareshape{cornerpoint}{
\inheritsavedanchors[from=rectangle] 
\inheritanchorborder[from=rectangle]
\inheritanchor[from=rectangle]{center}
\inheritanchor[from=rectangle]{north}
\inheritanchor[from=rectangle]{south}
\inheritanchor[from=rectangle]{west}
\inheritanchor[from=rectangle]{east}
\backgroundpath{
\southwest \pgf@xa=\pgf@x \pgf@ya=\pgf@y
\northeast \pgf@xb=\pgf@x \pgf@yb=\pgf@y

\pgfmathsetmacro{\pgf@shorten@left}{\pgfkeysvalueof{/tikz/shorten left}}
\pgfmathsetmacro{\pgf@shorten@right}{\pgfkeysvalueof{/tikz/shorten right}}

\pgfpathmoveto{\pgfpoint{0.5 * (\pgf@xa + \pgf@xb)}{\pgf@ya - 5pt}}
\pgfpathlineto{\pgfpoint{\pgf@xa - 8pt + \pgf@shorten@left}{\pgf@yb - 1.5 * \pgf@shorten@left}}
\pgfpathlineto{\pgfpoint{\pgf@xa - 8pt + \pgf@shorten@left}{\pgf@yb}}
\pgfpathlineto{\pgfpoint{\pgf@xb + 8pt - \pgf@shorten@right}{\pgf@yb}}
\pgfpathlineto{\pgfpoint{\pgf@xb + 8pt - \pgf@shorten@right}{\pgf@yb - 1.5 * \pgf@shorten@right}}
\pgfpathclose
}
}

\pgfdeclareshape{cornercopoint}{
\inheritsavedanchors[from=rectangle] 
\inheritanchorborder[from=rectangle]
\inheritanchor[from=rectangle]{center}
\inheritanchor[from=rectangle]{north}
\inheritanchor[from=rectangle]{south}
\inheritanchor[from=rectangle]{west}
\inheritanchor[from=rectangle]{east}
\backgroundpath{
\southwest \pgf@xa=\pgf@x \pgf@ya=\pgf@y
\northeast \pgf@xb=\pgf@x \pgf@yb=\pgf@y

\pgfmathsetmacro{\pgf@shorten@left}{\pgfkeysvalueof{/tikz/shorten left}}
\pgfmathsetmacro{\pgf@shorten@right}{\pgfkeysvalueof{/tikz/shorten right}}

\pgfpathmoveto{\pgfpoint{0.5 * (\pgf@xa + \pgf@xb)}{\pgf@yb + 5pt}}
\pgfpathlineto{\pgfpoint{\pgf@xa - 8pt + \pgf@shorten@left}{\pgf@ya + 1.5 * \pgf@shorten@left}}
\pgfpathlineto{\pgfpoint{\pgf@xa - 8pt + \pgf@shorten@left}{\pgf@ya}}
\pgfpathlineto{\pgfpoint{\pgf@xb + 8pt - \pgf@shorten@right}{\pgf@ya}}
\pgfpathlineto{\pgfpoint{\pgf@xb + 8pt - \pgf@shorten@right}{\pgf@ya + 1.5 * \pgf@shorten@right}}
\pgfpathclose
}
}

\makeatother

\pgfkeyssetvalue{/tikz/shorten left}{0pt}
\pgfkeyssetvalue{/tikz/shorten right}{0pt}

\tikzstyle{kpoint common}=[draw,fill=white,inner sep=1pt,minimum height=4mm]
\tikzstyle{kpoint}=[shape=cornerpoint,shorten left=5pt,kpoint common]
\tikzstyle{kpoint adjoint}=[shape=cornercopoint,shorten left=5pt,kpoint common]
\tikzstyle{kpoint conjugate}=[shape=cornerpoint,shorten right=5pt,kpoint common]
\tikzstyle{kpoint transpose}=[shape=cornercopoint,shorten right=5pt,kpoint common]
\tikzstyle{kpoint symm}=[shape=cornerpoint,shorten left=5pt,shorten right=5pt,kpoint common]

\tikzstyle{kpointdag}=[kpoint adjoint]
\tikzstyle{kpointadj}=[kpoint adjoint]
\tikzstyle{kpointconj}=[kpoint conjugate]
\tikzstyle{kpointtrans}=[kpoint transpose]

\tikzstyle{dkpoint}=[kpoint,doubled]
\tikzstyle{dkpointdag}=[kpoint adjoint,doubled]
\tikzstyle{dkpointadj}=[kpoint adjoint,doubled]
\tikzstyle{dkpointconj}=[kpoint conjugate,doubled]
\tikzstyle{dkpointtrans}=[kpoint transpose,doubled]

\tikzstyle{kscalar}=[kpoint common, shape=EBox, inner xsep=-1pt, inner ysep=3pt,font=\small]
\tikzstyle{kscalarconj}=[kpoint common, shape=WBox, inner xsep=-1pt, inner ysep=3pt,font=\small]

 \tikzstyle{upground}=[circuit ee IEC,thick,ground,rotate=90,scale=2.5]
 \tikzstyle{downground}=[circuit ee IEC,thick,ground,rotate=-90,scale=2.5]
 \tikzstyle{bigground}=[regular polygon,regular polygon sides=3,draw=gray,scale=0.50,inner sep=-0.5pt,minimum width=10mm,fill=gray]

\tikzstyle{arrs}=[-latex,font=\small,auto]
\tikzstyle{arrow plain}=[arrs]
\tikzstyle{arrow dashed}=[dashed,arrs]
\tikzstyle{arrow bold}=[very thick,arrs]
\tikzstyle{arrow hide}=[draw=white!0,-]
\tikzstyle{arrow reverse}=[latex-]
\tikzstyle{cdnode}=[]

\newcommand{\dotcounit}[1]{%
\,\begin{tikzpicture}[dotpic,yshift=-1mm]
\node [#1] (a) at (0,0.25) {}; 
\draw (0,-0.2)--(a);
\end{tikzpicture}\,}
\newcommand{\dotunit}[1]{%
\,\begin{tikzpicture}[dotpic,yshift=1.5mm]
\node [#1] (a) at (0,-0.25) {}; 
\draw (a)--(0,0.2);
\end{tikzpicture}\,}
\newcommand{\dotcomult}[1]{%
\,\begin{tikzpicture}[dotpic,yshift=0.5mm]
	\node [#1] (a) {};
	\draw (-90:0.55)--(a);
	\draw (a) -- (45:0.6);
	\draw (a) -- (135:0.6);
\end{tikzpicture}\,}
\newcommand{\dotmult}[1]{%
\,\begin{tikzpicture}[dotpic]
	\node [#1] (a) {};
	\draw (a) -- (90:0.55);
	\draw (a) (-45:0.6) -- (a);
	\draw (a) (-135:0.6) -- (a);
\end{tikzpicture}\,}

\newcommand{\prodmult}[2]{\dotmult{#1}\!\!\!\!\!\dotmult{#2}}

\newcommand{\unit}{\dotunit{black dot}}
\newcommand{\counit}{\dotcounit{black dot}}
\newcommand{\mult}{\dotmult{black dot}}
\newcommand{\comult}{\dotcomult{black dot}}

\newcommand{\whitemult}{\dotmult{white dot}}

\newcommand{\graymult}{\dotmult{gray dot}}

\let\olddagger\dagger
\renewcommand{\dagger}{\ensuremath{\olddagger}\xspace}

\theoremstyle{definition}
\newtheorem{theorem}{Theorem}[section]
\newtheorem{corollary}[theorem]{Corollary}

\newtheorem{definition}[theorem]{Definition}

\newtheorem{example*}[theorem]{Example*}
\newtheorem{examples*}[theorem]{Examples*}
\newtheorem{remark}[theorem]{Remark}
\newtheorem{remark*}[theorem]{Remark*}

\usepackage[color]{changebar}

\usepackage{color}
\def\bR{\begin{color}{red}} 
\def\bB{\begin{color}{blue}}
\def\bM{\begin{color}{magenta}}
\def\bC{\begin{color}{cyan}}
\def\bW{\begin{color}{white}}
\def\bBl{\begin{color}{black}} 
\def\bG{\begin{color}{green}}
\def\bY{\begin{color}{yellow}}
\def\e{\end{color}\xspace}
\newcommand{\bit}{\begin{itemize}}
\newcommand{\eit}{\end{itemize}\par\noindent}
\newcommand{\ben}{\begin{enumerate}}
\newcommand{\een}{\end{enumerate}\par\noindent}
\newcommand{\beq}{\begin{equation}}
\newcommand{\eeq}{\end{equation}\par\noindent}
\newcommand{\beqa}{\begin{eqnarray*}}
\newcommand{\eeqa}{\end{eqnarray*}\par\noindent}
\newcommand{\beqn}{\begin{eqnarray}}
\newcommand{\eeqn}{\end{eqnarray}\par\noindent}

\newcommand\dom{\textrm{dom}\xspace}
\newcommand\cod{\textrm{cod}\xspace}
\newcommand\DotCat{\textbf{Dot}\xspace}
\newcommand\Mat{\textbf{Mat}\xspace}
\newcommand\MGCat{\textbf{MGCat}\xspace}

\newcommand{\Free}{\textrm{Free}\xspace}
\newcommand{\Obj}{\textrm{Obj}\xspace}
\newcommand{\Mor}{\textrm{Mor}\xspace}
\newcommand{\Inputs}{\textrm{Inputs}\xspace}
\newcommand{\Outputs}{\textrm{Outputs}\xspace}
\newcommand{\Lbox}{\ell_b\xspace}
\newcommand{\Ldot}{\ell_d\xspace}
\newcommand{\Win}{\theta_{\textrm{in}}\xspace}
\newcommand{\Wout}{\theta_{\textrm{out}}}
\newcommand{\inp}{\textrm{in}}
\newcommand{\outp}{\textrm{out}}
\newcommand{\ZX}{\ensuremath{\mathbb Z[\mathbf{X}]}\xspace}
\newcommand{\ZXX}{\ensuremath{\mathbb Z[\mathbf{X} \cup \overline{\mathbf{X}}]}\xspace}

\title{Finite matrices are complete for (dagger-)hypergraph categories}
\author{Aleks Kissinger}

\begin{document}

\maketitle

\begin{abstract}
  Hypergraph categories are symmetric monoidal categories where each object is equipped with a special commutative Frobenius algebra (SCFA). Dagger-hypergraph categories are the same, but with dagger-symmetric monoidal categories and dagger-SCFAs. In this paper, we show that finite matrices over a field $K$ of characteristic $0$ are complete for hypergraph categories, and that finite matrices where $K$ has a non-trivial involution are complete for dagger-hypergraph categories.
\end{abstract}

\section{Introduction}

Hypergraph categories enrich the language of traced symmetric monoidal or compact closed categories by allowing many inputs and outputs of a morphism to be connected together. We can represent this in a minimal, algebraic way by equipping each object in the category with a special commutative Frobenius algebra (SCFA). Intuitively, SCFAs endow an object with the ability to `split', `merge', `initialise', and `terminate' a wire. Furthermore they have extremely well-behaved normal forms, in that any connected diagram is equal. Thus, in addition to being able to interpret normal string diagrams, we can naturally interpret their hypergraph variations, for example:
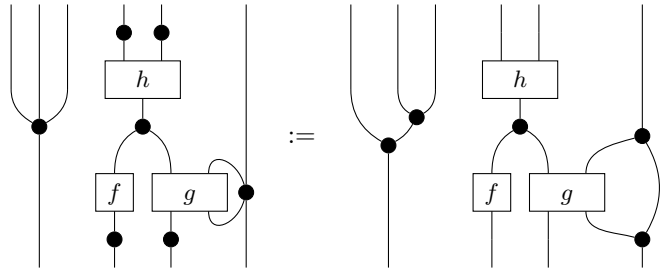
\[ \begin{tikzpicture}
	\begin{pgfonlayer}{nodelayer}
		\node [style=small box] (0) at (-0.75, -1.5) {$f$};
		\node [style=medium box] (1) at (0, 1.5) {$h$};
		\node [style=medium box] (2) at (1.25, -1.5) {$g$};
		\node [style=none] (3) at (0.75, -3.5) {};
		\node [style=black dot] (4) at (0.75, -2.75) {};
		\node [style=none] (5) at (0.75, -2) {};
		\node [style=none] (6) at (0.75, -1) {};
		\node [style=black dot] (7) at (0, 0.25) {};
		\node [style=none] (8) at (0, 1) {};
		\node [style=none] (9) at (2.75, -3.5) {};
		\node [style=black dot] (10) at (2.75, -1.5) {};
		\node [style=none] (11) at (2.75, 3.5) {};
		\node [style=none] (12) at (1.75, -1) {};
		\node [style=none] (13) at (1.75, -2) {};
		\node [style=none] (14) at (-0.5, 3.5) {};
		\node [style=none] (15) at (-0.5, 2) {};
		\node [style=black dot] (16) at (-0.5, 2.75) {};
		\node [style=none] (17) at (-0.75, -1) {};
		\node [style=none] (18) at (-0.75, -2) {};
		\node [style=none] (19) at (-0.75, -3.5) {};
		\node [style=black dot] (20) at (-0.75, -2.75) {};
		\node [style=black dot] (21) at (-2.75, 0.25) {};
		\node [style=none] (22) at (-2.75, -3.5) {};
		\node [style=none] (23) at (-3.5, 1.25) {};
		\node [style=none] (24) at (-3.5, 3.5) {};
		\node [style=none] (25) at (-2, 1.25) {};
		\node [style=none] (26) at (-2, 3.5) {};
		\node [style=none] (27) at (-2.75, 3.5) {};
		\node [style=none] (28) at (0.5, 3.5) {};
		\node [style=none] (29) at (0.5, 2) {};
		\node [style=black dot] (30) at (0.5, 2.75) {};
	\end{pgfonlayer}
	\begin{pgfonlayer}{edgelayer}
		\draw (3.center) to (4);
		\draw (4) to (5.center);
		\draw [in=-30, out=90, looseness=1.00] (6.center) to (7);
		\draw (7) to (8.center);
		\draw (9.center) to (10);
		\draw (10) to (11.center);
		\draw [in=105, out=90, looseness=1.75] (12.center) to (10);
		\draw [in=-90, out=-105, looseness=1.75] (10) to (13.center);
		\draw (15.center) to (16);
		\draw (16) to (14.center);
		\draw [in=-150, out=90, looseness=1.00] (17.center) to (7);
		\draw (19.center) to (20);
		\draw (20) to (18.center);
		\draw (22.center) to (21);
		\draw [in=-90, out=150, looseness=1.00] (21) to (23.center);
		\draw [in=-90, out=30, looseness=1.00] (21) to (25.center);
		\draw (25.center) to (26.center);
		\draw (23.center) to (24.center);
		\draw (21) to (27.center);
		\draw (29.center) to (30);
		\draw (30) to (28.center);
	\end{pgfonlayer}
\end{tikzpicture} \quad := \quad 
   \begin{tikzpicture}
	\begin{pgfonlayer}{nodelayer}
		\node [style=small box] (0) at (-0.75, -1.5) {$f$};
		\node [style=medium box] (1) at (0, 1.5) {$h$};
		\node [style=medium box] (2) at (1.25, -1.5) {$g$};
		\node [style=none] (3) at (0.75, -3.5) {};
		\node [style=none] (4) at (0.75, -2.75) {};
		\node [style=none] (5) at (0.75, -2) {};
		\node [style=none] (6) at (0.75, -1) {};
		\node [style=black dot] (7) at (0, 0.25) {};
		\node [style=none] (8) at (0, 1) {};
		\node [style=none] (9) at (3.25, -3.5) {};
		\node [style=none] (10) at (3.25, 3.5) {};
		\node [style=none] (11) at (1.75, -1) {};
		\node [style=none] (12) at (1.75, -2) {};
		\node [style=none] (13) at (-0.5, 3.5) {};
		\node [style=none] (14) at (-0.5, 2) {};
		\node [style=none] (15) at (-0.5, 2.75) {};
		\node [style=none] (16) at (-0.75, -1) {};
		\node [style=none] (17) at (-0.75, -2) {};
		\node [style=none] (18) at (-0.75, -3.5) {};
		\node [style=none] (19) at (-0.75, -2.75) {};
		\node [style=none] (20) at (-3.5, -3.5) {};
		\node [style=none] (21) at (-4.5, 1.25) {};
		\node [style=none] (22) at (-4.5, 3.5) {};
		\node [style=none] (23) at (-2.25, 1.25) {};
		\node [style=none] (24) at (-2.25, 3.5) {};
		\node [style=none] (25) at (-3.25, 3.5) {};
		\node [style=none] (26) at (0.5, 3.5) {};
		\node [style=none] (27) at (0.5, 2) {};
		\node [style=none] (28) at (0.5, 2.75) {};
		\node [style=black dot] (29) at (-2.75, 0.5) {};
		\node [style=black dot] (30) at (-3.5, -0.25) {};
		\node [style=none] (31) at (-3.25, 1.25) {};
		\node [style=black dot] (32) at (3.25, -2.75) {};
		\node [style=black dot] (33) at (3.25, 0) {};
	\end{pgfonlayer}
	\begin{pgfonlayer}{edgelayer}
		\draw (3.center) to (4.center);
		\draw (4.center) to (5.center);
		\draw [in=-30, out=90, looseness=1.00] (6.center) to (7);
		\draw (7) to (8.center);
		\draw (14.center) to (15.center);
		\draw (15.center) to (13.center);
		\draw [in=-150, out=90, looseness=1.00] (16.center) to (7);
		\draw (18.center) to (19.center);
		\draw (19.center) to (17.center);
		\draw (23.center) to (24.center);
		\draw (21.center) to (22.center);
		\draw (27.center) to (28.center);
		\draw (28.center) to (26.center);
		\draw (31.center) to (25.center);
		\draw [in=-90, out=30, looseness=1.00] (29) to (23.center);
		\draw [in=-90, out=150, looseness=1.00] (29) to (31.center);
		\draw [bend right=15, looseness=1.00] (30) to (29);
		\draw [in=-90, out=150, looseness=1.00] (30) to (21.center);
		\draw (9.center) to (32);
		\draw [in=-90, out=165, looseness=1.00] (32) to (12.center);
		\draw [bend right, looseness=1.00] (32) to (33);
		\draw [in=-165, out=90, looseness=1.00] (11.center) to (33);
		\draw (33) to (10.center);
		\draw (20.center) to (30);
	\end{pgfonlayer}
\end{tikzpicture} \]

These structures play an important role in categories whose morphisms can be written as matrices over unital semirings, though they were perhaps overlooked for quite some time due to a certain failure of naturality, which concretely manifests itself as a `basis dependence'. However, in 2006 Coecke, Pavlovic, and Vicary pointed out that this basis-dependence is no accident, but rather that SCFAs \textit{characterise} bases in categories of linear maps~\cite{CPV}, as well as provide a useful set of building blocks for the types of maps one would define using a basis. This feature has been exploited numerous times, particularly within the program of categorical quantum mechanics~\cite{AC1,CD2,CPStar}.

This paper adds another piece to the story connecting SCFAs and bases. Namely, that the axioms of dagger-hypergraph categories are \textit{complete} for the category of finite-dimensional vector spaces where each object has a chosen basis, i.e. the category of \textit{finite matrices}.


\section{Preliminaries}\label{sec:prelims}

\begin{definition}
  A symmetric monoidal category $\mathcal C$ is called a \textit{dagger-symmetric monoidal category} if there exists a monoidal functor $(-)^\dagger : \mathcal C \to \mathcal C^{\textrm{op}}$ such that $\dagger \circ \dagger = 1_{\mathcal C}$ and:
  \[ \alpha^{-1} = \alpha^{\dagger} \qquad\qquad
     \lambda^{-1} = \lambda^{\dagger} \qquad\qquad
     \gamma^{-1} = \gamma^{\dagger} \]
  for $\alpha, \lambda, \gamma$ the associativity, unit, and symmetry natural transformations of $\mathcal C$.
\end{definition}

\begin{definition}
  A special commutative Frobenius algebra (SCFA) in a symmetric monoidal category $\mathcal C$ consists of a tuple $(A, \mu, \eta, \delta, \epsilon)$ satisfying the following equations:
  \[
      \mu \circ (1_A \otimes \eta) = 1_A \qquad\quad
      \mu \circ (1_A \otimes \mu) = \mu \circ (\mu \otimes 1_A) \qquad\quad
      \mu \circ \gamma_{A,A} = \mu
      \tag{\bf monoid}
  \]

  \[
      (1_A \otimes \epsilon) \circ \delta = 1_A \qquad\quad
      (1_A \otimes \delta) \circ \delta = (\delta \otimes 1_A) \circ \delta \qquad\quad
      \gamma_{A,A} \circ \delta = \delta
      \tag{\bf comonoid}
  \]

  \[
      (1_A \otimes \mu) \circ (\delta \otimes 1_A) = \delta \circ \mu
      \tag{\bf frobenius}
  \]

  \[
      \mu \circ \delta = 1_A
      \tag{\bf special}
  \]
  If $\mathcal C$ is a dagger-symmetric monoidal category, a \textit{dagger-SCFA} (dSCFA) additionally satisfies the equations
  \[ \mu^\dagger = \delta \qquad\qquad \eta^\dagger = \epsilon
     \tag{\bf dagger} \]
\end{definition}

From hence forth, we will always be working in a symmetric moniodal category, so we will use \textit{string diagram} notation for morphisms. In diagrammatic notation, the axioms of an SCFA become:
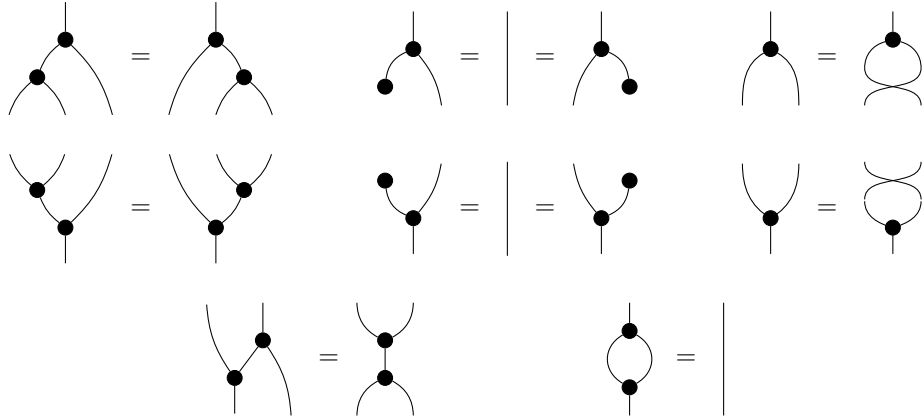
\ctikzfig{scfa-axioms}
The following is a well-known folk theorem about SCFAs.

\begin{theorem}[Spider] \label{thm:spider}
  Suppose $f$ and $g$ can be written as a connected string diagram from $A^{\otimes m}$ to $A^{\otimes n}$ consisting just of the morphisms from a single SCFA, then $f = g$.
\end{theorem}

\begin{proof}
  The proof proceeds by induction on the size of diagrams, showing that any connected diagram can be rewritten into a canonical diagram consisting of a tree of multiplies followed by an upside-down tree of comultiplies:
  
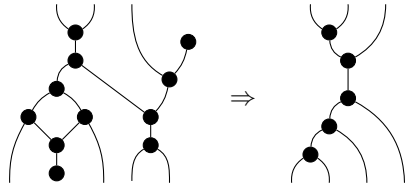
\[ \begin{tikzpicture}
	\begin{pgfonlayer}{nodelayer}
		\node [style=black dot] (0) at (-0.5, 1.75) {};
		\node [style=none] (1) at (-1, 2.5) {};
		\node [style=none] (2) at (0, 2.5) {};
		\node [style=none] (3) at (1, 2.5) {};
		\node [style=none] (4) at (-2.25, -2.25) {};
		\node [style=none] (5) at (0.25, -2.25) {};
		\node [style=none] (6) at (1, -2.25) {};
		\node [style=none] (7) at (2, -2.25) {};
		\node [style=black dot] (8) at (-0.5, 1) {};
		\node [style=black dot] (9) at (-1, 0.25) {};
		\node [style=black dot] (10) at (1.5, -0.5) {};
		\node [style=black dot] (11) at (2, 0.5) {};
		\node [style=black dot] (12) at (2.5, 1.5) {};
		\node [style=black dot] (13) at (1.5, -1.25) {};
		\node [style=black dot] (14) at (-1, -1.25) {};
		\node [style=black dot] (15) at (-0.25, -0.5) {};
		\node [style=black dot] (16) at (-1.75, -0.5) {};
		\node [style=black dot] (17) at (-1, -2) {};
	\end{pgfonlayer}
	\begin{pgfonlayer}{edgelayer}
		\draw [bend left, looseness=1.00] (0) to (1.center);
		\draw [bend right, looseness=1.00] (0) to (2.center);
		\draw (8) to (0);
		\draw [bend left, looseness=1.00] (9) to (8);
		\draw [bend right=15, looseness=1.00] (11) to (12);
		\draw [in=-90, out=150, looseness=1.00] (11) to (3.center);
		\draw [bend right=15, looseness=1.00] (10) to (11);
		\draw (10) to (8);
		\draw (10) to (13);
		\draw [in=-150, out=90, looseness=1.00] (6) to (13);
		\draw [in=-30, out=90, looseness=1.00] (7.center) to (13);
		\draw [bend right=15, looseness=1.00] (15) to (9);
		\draw [bend right=15, looseness=1.00] (5.center) to (15);
		\draw (14) to (15);
		\draw [bend left=15, looseness=1.00] (16) to (9);
		\draw [bend left=15, looseness=1.00] (4.center) to (16);
		\draw (17) to (14);
		\draw (14) to (16);
	\end{pgfonlayer}
\end{tikzpicture} \quad \Rightarrow \quad \begin{tikzpicture}
	\begin{pgfonlayer}{nodelayer}
		\node [style=black dot] (0) at (-1, -1.5) {};
		\node [style=black dot] (1) at (-0.5, -0.75) {};
		\node [style=black dot] (2) at (0, 0) {};
		\node [style=black dot] (3) at (0, 1) {};
		\node [style=black dot] (4) at (-0.5, 1.75) {};
		\node [style=none] (5) at (-1, 2.5) {};
		\node [style=none] (6) at (0, 2.5) {};
		\node [style=none] (7) at (1, 2.5) {};
		\node [style=none] (8) at (-1.5, -2.25) {};
		\node [style=none] (9) at (-0.5, -2.25) {};
		\node [style=none] (10) at (0.5, -2.25) {};
		\node [style=none] (11) at (1.5, -2.25) {};
	\end{pgfonlayer}
	\begin{pgfonlayer}{edgelayer}
		\draw [bend left, looseness=1.00] (0) to (1);
		\draw [bend left, looseness=1.00] (1) to (2);
		\draw (2) to (3);
		\draw [bend left, looseness=1.00] (3) to (4);
		\draw [bend left, looseness=1.00] (4) to (5.center);
		\draw [bend right, looseness=1.00] (4) to (6.center);
		\draw [bend right, looseness=1.00] (3) to (7.center);
		\draw [bend right, looseness=1.00] (0) to (8.center);
		\draw [bend left, looseness=1.00] (0) to (9.center);
		\draw [bend left, looseness=1.00] (1) to (10.center);
		\draw [bend left, looseness=1.00] (2) to (11.center);
	\end{pgfonlayer}
\end{tikzpicture} \]
  See e.g.~\cite{AleksThesis}, Theorem 3.2.28. For an alternative derivation via distributive laws see~\cite{Lack}, \S5.4.
\end{proof}

We express the fact that there is only one connected diagram with $m$ inputs and $n$ outputs by collapsing connected diagrams into a single dot:
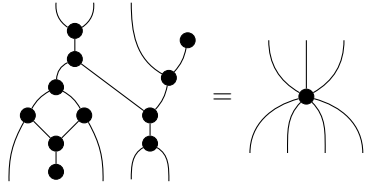
\[
\begin{tikzpicture}
	\begin{pgfonlayer}{nodelayer}
		\node [style=black dot] (0) at (-0.5, 1.75) {};
		\node [style=none] (1) at (-1, 2.5) {};
		\node [style=none] (2) at (0, 2.5) {};
		\node [style=none] (3) at (1, 2.5) {};
		\node [style=none] (4) at (-2.25, -2.25) {};
		\node [style=none] (5) at (0.25, -2.25) {};
		\node [style=none] (6) at (1, -2.25) {};
		\node [style=none] (7) at (2, -2.25) {};
		\node [style=black dot] (8) at (-0.5, 1) {};
		\node [style=black dot] (9) at (-1, 0.25) {};
		\node [style=black dot] (10) at (1.5, -0.5) {};
		\node [style=black dot] (11) at (2, 0.5) {};
		\node [style=black dot] (12) at (2.5, 1.5) {};
		\node [style=black dot] (13) at (1.5, -1.25) {};
		\node [style=black dot] (14) at (-1, -1.25) {};
		\node [style=black dot] (15) at (-0.25, -0.5) {};
		\node [style=black dot] (16) at (-1.75, -0.5) {};
		\node [style=black dot] (17) at (-1, -2) {};
	\end{pgfonlayer}
	\begin{pgfonlayer}{edgelayer}
		\draw [bend left, looseness=1.00] (0) to (1.center);
		\draw [bend right, looseness=1.00] (0) to (2.center);
		\draw (8) to (0);
		\draw [bend left, looseness=1.00] (9) to (8);
		\draw [bend right=15, looseness=1.00] (11) to (12);
		\draw [in=-90, out=150, looseness=1.00] (11) to (3.center);
		\draw [bend right=15, looseness=1.00] (10) to (11);
		\draw (10) to (8);
		\draw (10) to (13);
		\draw [in=-150, out=90, looseness=1.00] (6) to (13);
		\draw [in=-30, out=90, looseness=1.00] (7.center) to (13);
		\draw [bend right=15, looseness=1.00] (15) to (9);
		\draw [bend right=15, looseness=1.00] (5.center) to (15);
		\draw (14) to (15);
		\draw [bend left=15, looseness=1.00] (16) to (9);
		\draw [bend left=15, looseness=1.00] (4.center) to (16);
		\draw (17) to (14);
		\draw (14) to (16);
	\end{pgfonlayer}
\end{tikzpicture} \ =\ \begin{tikzpicture}
	\begin{pgfonlayer}{nodelayer}
		\node [style=black dot] (0) at (0, 0) {};
		\node [style=none] (1) at (-1, 1.5) {};
		\node [style=none] (2) at (0, 1.5) {};
		\node [style=none] (3) at (1, 1.5) {};
		\node [style=none] (4) at (-1.5, -1.5) {};
		\node [style=none] (5) at (-0.5, -1.5) {};
		\node [style=none] (6) at (0.5, -1.5) {};
		\node [style=none] (7) at (1.5, -1.5) {};
	\end{pgfonlayer}
	\begin{pgfonlayer}{edgelayer}
		\draw [in=90, out=-15, looseness=1.00] (0) to (7.center);
		\draw [in=-165, out=90, looseness=1.00] (4.center) to (0);
		\draw [in=-135, out=90, looseness=1.00] (5.center) to (0);
		\draw [in=-45, out=90, looseness=1.00] (6.center) to (0);
		\draw [in=-90, out=30, looseness=1.00] (0) to (3.center);
		\draw (0) to (2.center);
		\draw [in=-90, out=150, looseness=1.00] (0) to (1.center);
	\end{pgfonlayer}
\end{tikzpicture}
\]

\begin{definition}\label{def:spider}
  Canonical tree/co-tree morphisms are called \textit{spiders}. We let:
  \[ S_m^n \ := \ 
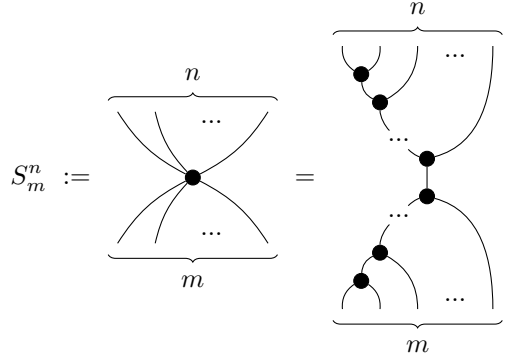
\begin{tikzpicture}
	\begin{pgfonlayer}{nodelayer}
		\node [style=none] (0) at (-2.25, -1.75) {};
		\node [style=none] (1) at (-1.25, -1.75) {};
		\node [style=none] (2) at (1.75, -1.75) {};
		\node [style=none] (3) at (0.25, -1.5) {...};
		\node [style=none] (4) at (-1.25, 1.75) {};
		\node [style=none] (5) at (0.25, 1.5) {...};
		\node [style=none] (6) at (1.75, 1.75) {};
		\node [style=none] (7) at (-2.25, 1.75) {};
		\node [style=none] (8) at (-2.5, 2.25) {};
		\node [style=none] (9) at (2, 2.25) {};
		\node [style=none] (10) at (2, -2.25) {};
		\node [style=none] (11) at (-2.5, -2.25) {};
		\node [style=none] (12) at (-0.25, -2.75) {$m$};
		\node [style=none] (13) at (-0.25, 2.75) {$n$};
		\node [style=black dot] (14) at (-0.25, 0) {};
	\end{pgfonlayer}
	\begin{pgfonlayer}{edgelayer}
		\draw [style=small braceedge] (8.center) to (9.center);
		\draw [style=small braceedge] (10.center) to (11.center);
		\draw [bend left=15, looseness=1.00] (0.center) to (14);
		\draw [bend left=15, looseness=1.00] (1.center) to (14);
		\draw [bend right=15, looseness=1.00] (2.center) to (14);
		\draw [bend left=15, looseness=1.00] (14) to (7.center);
		\draw [bend left=15, looseness=1.00] (14) to (4.center);
		\draw [bend right=15, looseness=1.00] (14) to (6.center);
	\end{pgfonlayer}
\end{tikzpicture}\ =\ \begin{tikzpicture}
	\begin{pgfonlayer}{nodelayer}
		\node [style=black dot] (0) at (-1.75, -2.75) {};
		\node [style=black dot] (1) at (-1.25, -2) {};
		\node [style=black dot] (2) at (0, -0.5) {};
		\node [style=none] (3) at (-2.25, -3.5) {};
		\node [style=none] (4) at (-1.25, -3.5) {};
		\node [style=none] (5) at (-0.25, -3.5) {};
		\node [style=none] (6) at (1.75, -3.5) {};
		\node [style=none] (7) at (-0.5, -0.75) {};
		\node [style=none] (8) at (-1, -1.25) {};
		\node [style=none] (9) at (-0.75, -1) {...};
		\node [style=none] (10) at (0.75, -3.25) {...};
		\node [style=none] (11) at (-0.5, 0.75) {};
		\node [style=none] (12) at (-0.75, 1) {...};
		\node [style=black dot] (13) at (-1.25, 2) {};
		\node [style=black dot] (14) at (-1.75, 2.75) {};
		\node [style=none] (15) at (-1.25, 3.5) {};
		\node [style=black dot] (16) at (0, 0.5) {};
		\node [style=none] (17) at (0.75, 3.25) {...};
		\node [style=none] (18) at (1.75, 3.5) {};
		\node [style=none] (19) at (-2.25, 3.5) {};
		\node [style=none] (20) at (-1, 1.25) {};
		\node [style=none] (21) at (-0.25, 3.5) {};
		\node [style=none] (22) at (-2.5, 4) {};
		\node [style=none] (23) at (2, 4) {};
		\node [style=none] (24) at (2, -4) {};
		\node [style=none] (25) at (-2.5, -4) {};
		\node [style=none] (26) at (-0.25, -4.5) {$m$};
		\node [style=none] (27) at (-0.25, 4.5) {$n$};
	\end{pgfonlayer}
	\begin{pgfonlayer}{edgelayer}
		\draw [bend left, looseness=1.00] (0) to (1);
		\draw [bend right, looseness=1.00] (0) to (3.center);
		\draw [bend left, looseness=1.00] (0) to (4.center);
		\draw [bend left, looseness=1.00] (1) to (5.center);
		\draw [in=90, out=-15, looseness=1.00] (2) to (6.center);
		\draw [in=-135, out=90, looseness=0.75] (1) to (8.center);
		\draw [in=-165, out=45, looseness=0.75] (7.center) to (2);
		\draw [bend right, looseness=1.00] (14) to (13);
		\draw [bend left, looseness=1.00] (14) to (19.center);
		\draw [bend right, looseness=1.00] (14) to (15.center);
		\draw [bend right, looseness=1.00] (13) to (21.center);
		\draw [in=-90, out=15, looseness=1.00] (16) to (18.center);
		\draw [in=135, out=-90, looseness=0.75] (13) to (20.center);
		\draw [in=165, out=-45, looseness=0.75] (11.center) to (16);
		\draw (2) to (16);
		\draw [style=small braceedge] (22.center) to (23.center);
		\draw [style=small braceedge] (24.center) to (25.center);
	\end{pgfonlayer}
\end{tikzpicture} \]
\end{definition}

\section{(Dagger-)hypergraph categories}\label{sec:hypergraph-cats}

\begin{definition}
  A symmetric monoidal category $\mathcal C$ where every object is equipped with an SCFA satisfying:
  \[ (A, \whitemult) \otimes (B, \graymult) \ =\ (A \otimes B, \prodmult{white dot}{gray dot}) \]
  is called a \textit{hypergraph category}. Similarly, a dagger-symmetric monoidal category where every object is equipped with a dSCFA is called a \textit{dagger-hypergraph category}.
\end{definition}

A strong (dagger-)hypergraph functor is a strong (dagger-)symmetric monoidal functor that preserves the hypergraph structure. A strong monoidal functor is triple $(F, p_F, U_F)$ consisting of a functor $F : \mathcal C \to \mathcal D$ and natural isomorphisms:
\[ (p_F)_{A,B} : F(A \otimes B) \to FA \otimes FB  \qquad\qquad u_F : FI \to I \]
satisfying:
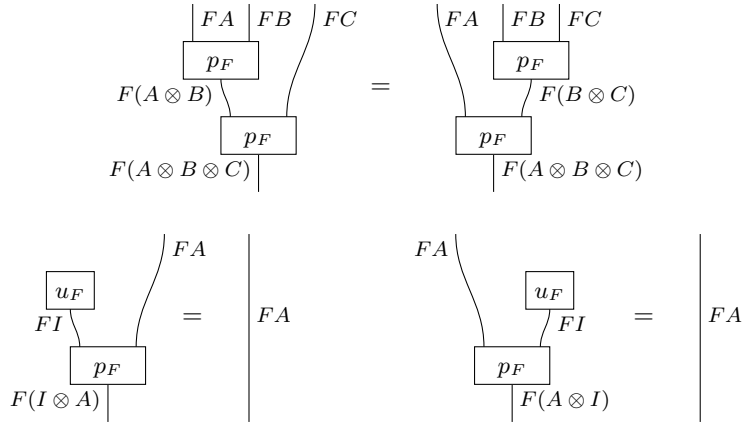
\[
\begin{tikzpicture}
	\begin{pgfonlayer}{nodelayer}
		\node [style=medium box] (0) at (-3.25, -1.25) {$p_F$};
		\node [style=none] (1) at (-3.25, -1.75) {};
		\node [style=none] (2) at (-3.25, -2.75) {};
		\node [style=none] (3) at (-4, -0.75) {};
		\node [style=none] (4) at (-4.25, 0.25) {};
		\node [style=none] (5) at (-2.5, -0.75) {};
		\node [style=none] (6) at (-1.75, 2.25) {};
		\node [style=left label] (7) at (-4.5, -0.25) {$F(A\otimes B)$};
		\node [style=right label] (8) at (-1.5, 1.75) {$FC$};
		\node [style=left label] (9) at (-3.5, -2.25) {$F(A\otimes B \otimes C)$};
		\node [style=none] (10) at (-5, 2.25) {};
		\node [style=none] (11) at (-5, 1.25) {};
		\node [style=medium box] (12) at (-4.25, 0.75) {$p_F$};
		\node [style=right label] (13) at (-4.75, 1.75) {$FA$};
		\node [style=none] (14) at (-3.5, 1.25) {};
		\node [style=right label] (15) at (-3.25, 1.75) {$FB$};
		\node [style=none] (16) at (-3.5, 2.25) {};
		\node [style=none] (17) at (0, 0) {$=$};
		\node [style=right label] (18) at (1.75, 1.75) {$FA$};
		\node [style=right label] (19) at (3.25, -2.25) {$F(A\otimes B \otimes C)$};
		\node [style=none] (20) at (3, -2.75) {};
		\node [style=medium box] (21) at (4, 0.75) {$p_F$};
		\node [style=right label] (22) at (4.25, -0.25) {$F(B\otimes C)$};
		\node [style=none] (23) at (4, 0.25) {};
		\node [style=right label] (24) at (5, 1.75) {$FC$};
		\node [style=none] (25) at (4.75, 2.25) {};
		\node [style=right label] (26) at (3.5, 1.75) {$FB$};
		\node [style=none] (27) at (3.25, 1.25) {};
		\node [style=none] (28) at (2.25, -0.75) {};
		\node [style=none] (29) at (1.5, 2.25) {};
		\node [style=none] (30) at (3, -1.75) {};
		\node [style=none] (31) at (4.75, 1.25) {};
		\node [style=none] (32) at (3.75, -0.75) {};
		\node [style=none] (33) at (3.25, 2.25) {};
		\node [style=medium box] (34) at (3, -1.25) {$p_F$};
	\end{pgfonlayer}
	\begin{pgfonlayer}{edgelayer}
		\draw (2.center) to (1.center);
		\draw [in=-90, out=90, looseness=1.00] (3.center) to (4.center);
		\draw [in=-90, out=90, looseness=1.00] (5.center) to (6.center);
		\draw (11.center) to (10.center);
		\draw (14.center) to (16.center);
		\draw (20.center) to (30.center);
		\draw [in=-90, out=90, looseness=1.00] (32.center) to (23.center);
		\draw [in=-90, out=90, looseness=1.00] (28.center) to (29.center);
		\draw (31.center) to (25.center);
		\draw (27.center) to (33.center);
	\end{pgfonlayer}
\end{tikzpicture}
\]

\medskip

\[
\begin{tikzpicture}
	\begin{pgfonlayer}{nodelayer}
		\node [style=medium box] (0) at (-2.25, -1.25) {$p_F$};
		\node [style=none] (1) at (-2.25, -1.75) {};
		\node [style=none] (2) at (-2.25, -2.75) {};
		\node [style=none] (3) at (-3, -0.75) {};
		\node [style=none] (4) at (-3.25, 0.25) {};
		\node [style=none] (5) at (-1.5, -0.75) {};
		\node [style=none] (6) at (-0.75, 2.25) {};
		\node [style=right label] (7) at (-0.5, 1.75) {$FA$};
		\node [style=left label] (8) at (-2.5, -2.25) {$F(I \otimes A)$};
		\node [style=small box] (9) at (-3.25, 0.75) {$u_F$};
		\node [style=none] (10) at (0, 0) {$=$};
		\node [style=right label] (11) at (1.75, 0) {$FA$};
		\node [style=none] (12) at (1.5, 2.25) {};
		\node [style=none] (13) at (1.5, -2.75) {};
		\node [style=left label] (14) at (-3.5, -0.25) {$FI$};
	\end{pgfonlayer}
	\begin{pgfonlayer}{edgelayer}
		\draw (2.center) to (1.center);
		\draw [in=-90, out=90, looseness=1.00] (3.center) to (4.center);
		\draw [in=-90, out=90, looseness=1.00] (5.center) to (6.center);
		\draw (13.center) to (12.center);
	\end{pgfonlayer}
\end{tikzpicture}
\qquad\qquad
\begin{tikzpicture}
	\begin{pgfonlayer}{nodelayer}
		\node [style=medium box] (0) at (-3.5, -1.25) {$p_F$};
		\node [style=none] (1) at (-3.5, -1.75) {};
		\node [style=none] (2) at (-3.5, -2.75) {};
		\node [style=none] (3) at (-2.75, -0.75) {};
		\node [style=none] (4) at (-2.5, 0.25) {};
		\node [style=none] (5) at (-4.25, -0.75) {};
		\node [style=none] (6) at (-5, 2.25) {};
		\node [style=left label] (7) at (-5.25, 1.75) {$FA$};
		\node [style=right label] (8) at (-3.25, -2.25) {$F(A \otimes I)$};
		\node [style=small box] (9) at (-2.5, 0.75) {$u_F$};
		\node [style=none] (10) at (0, 0) {$=$};
		\node [style=right label] (11) at (1.75, 0) {$FA$};
		\node [style=none] (12) at (1.5, 2.25) {};
		\node [style=none] (13) at (1.5, -2.75) {};
		\node [style=right label] (14) at (-2.25, -0.25) {$FI$};
	\end{pgfonlayer}
	\begin{pgfonlayer}{edgelayer}
		\draw (2.center) to (1.center);
		\draw [in=-90, out=90, looseness=1.00] (3.center) to (4.center);
		\draw [in=-90, out=90, looseness=1.00] (5.center) to (6.center);
		\draw (13.center) to (12.center);
	\end{pgfonlayer}
\end{tikzpicture}
\]
Preserving symmetries can be written using `functorial box' notation (c.f.~\cite{PaulAndreFunBox}):
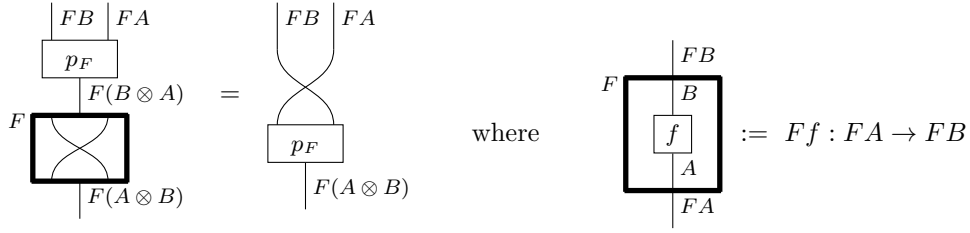
\[
\begin{tikzpicture}
	\begin{pgfonlayer}{nodelayer}
		\node [style=medium box] (0) at (-1, 2) {$p_F$};
		\node [style=right label] (1) at (-0.75, 1) {$F(B \otimes A)$};
		\node [style=none] (2) at (-1, 0.5) {};
		\node [style=none] (3) at (-1, 1.5) {};
		\node [style=none] (4) at (-1.75, 2.5) {};
		\node [style=none] (5) at (-1.75, 3.5) {};
		\node [style=right label] (6) at (0, 3) {$FA$};
		\node [style=none] (7) at (-0.25, 2.5) {};
		\node [style=right label] (8) at (-1.5, 3) {$FB$};
		\node [style=none] (9) at (-0.25, 3.5) {};
		\node [style=none] (10) at (-2.25, 0.5) {};
		\node [style=none] (11) at (0.25, 0.5) {};
		\node [style=none] (12) at (0.25, -1.25) {};
		\node [style=none] (13) at (-2.25, -1.25) {};
		\node [style=left label] (14) at (-2.5, 0.25) {$F$};
		\node [style=right label] (15) at (-0.75, -1.75) {$F(A \otimes B)$};
		\node [style=none] (16) at (-1, -2.25) {};
		\node [style=none] (17) at (-1, -1.25) {};
		\node [style=none] (18) at (-0.25, -1.25) {};
		\node [style=none] (19) at (-1.75, 0.5) {};
		\node [style=none] (20) at (-1.75, -1.25) {};
		\node [style=none] (21) at (-0.25, 0.5) {};
		\node [style=none] (22) at (3, 1) {$=$};
		\node [style=none] (23) at (4.25, 0.25) {};
		\node [style=none] (24) at (5.75, 0.25) {};
		\node [style=none] (25) at (4.25, 2.25) {};
		\node [style=none] (26) at (5.75, 2.25) {};
		\node [style=right label] (27) at (5.25, -1.5) {$F(A \otimes B)$};
		\node [style=none] (28) at (5, -2) {};
		\node [style=none] (29) at (5, -0.75) {};
		\node [style=medium box] (30) at (5, -0.25) {$p_F$};
		\node [style=right label] (31) at (6, 3) {$FA$};
		\node [style=none] (32) at (4.25, 3.5) {};
		\node [style=none] (33) at (5.75, 3.5) {};
		\node [style=right label] (34) at (4.5, 3) {$FB$};
	\end{pgfonlayer}
	\begin{pgfonlayer}{edgelayer}
		\draw (2.center) to (3.center);
		\draw (4.center) to (5.center);
		\draw (7.center) to (9.center);
		\draw [style=boldedge] (10.center) to (11.center);
		\draw [style=boldedge] (11.center) to (12.center);
		\draw [style=boldedge] (12.center) to (13.center);
		\draw [style=boldedge] (13.center) to (10.center);
		\draw (16.center) to (17.center);
		\draw [in=-90, out=90, looseness=1.00] (18.center) to (19.center);
		\draw [in=-90, out=90, looseness=1.00] (20.center) to (21.center);
		\draw [in=-90, out=90, looseness=1.00] (24.center) to (25.center);
		\draw [in=-90, out=90, looseness=1.00] (23.center) to (26.center);
		\draw (28.center) to (29.center);
		\draw (25.center) to (32.center);
		\draw (26.center) to (33.center);
	\end{pgfonlayer}
\end{tikzpicture}
\qquad\textrm{where}\qquad
\begin{tikzpicture}
	\begin{pgfonlayer}{nodelayer}
		\node [style=right label] (0) at (0.25, 2) {$FB$};
		\node [style=none] (1) at (0, 1.5) {};
		\node [style=none] (2) at (0, 2.5) {};
		\node [style=none] (3) at (-1.25, 1.5) {};
		\node [style=none] (4) at (1.25, 1.5) {};
		\node [style=none] (5) at (1.25, -1.5) {};
		\node [style=none] (6) at (-1.25, -1.5) {};
		\node [style=left label] (7) at (-1.5, 1.25) {$F$};
		\node [style=right label] (8) at (0.25, -2) {$FA$};
		\node [style=none] (9) at (0, -2.5) {};
		\node [style=none] (10) at (0, -1.5) {};
		\node [style=none] (11) at (0, -1.5) {};
		\node [style=none] (12) at (0, 1.5) {};
		\node [style=small box] (13) at (0, 0) {$f$};
		\node [style=right label] (14) at (0.25, -1) {$A$};
		\node [style=right label] (15) at (0.25, 1) {$B$};
	\end{pgfonlayer}
	\begin{pgfonlayer}{edgelayer}
		\draw (1.center) to (2.center);
		\draw [style=boldedge] (3.center) to (4.center);
		\draw [style=boldedge] (4.center) to (5.center);
		\draw [style=boldedge] (5.center) to (6.center);
		\draw [style=boldedge] (6.center) to (3.center);
		\draw (9.center) to (10.center);
		\draw (11.center) to (13);
		\draw (13) to (12.center);
	\end{pgfonlayer}
\end{tikzpicture} \ :=\ Ff : FA \to FB
\]
A strong dagger-symmetric monoidal functor additionally satisfies $F(f^\dagger) = (Ff)^\dagger$, and preserving the hypergraph structure means:
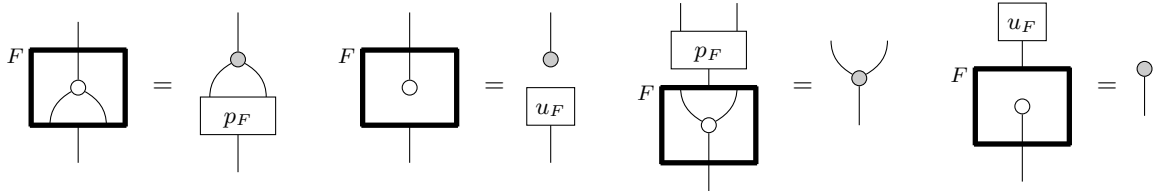
\[
\begin{tikzpicture}
	\begin{pgfonlayer}{nodelayer}
		\node [style=none] (0) at (-2, 1) {};
		\node [style=none] (1) at (-2, 1.75) {};
		\node [style=none] (2) at (-3.25, 1) {};
		\node [style=none] (3) at (-0.75, 1) {};
		\node [style=none] (4) at (-0.75, -1) {};
		\node [style=none] (5) at (-3.25, -1) {};
		\node [style=left label] (6) at (-3.5, 0.75) {$F$};
		\node [style=none] (7) at (-2, -2) {};
		\node [style=none] (8) at (-2, -1) {};
		\node [style=none] (9) at (0.25, 0) {$=$};
		\node [style=none] (10) at (1.5, -0.25) {};
		\node [style=none] (11) at (3, -0.25) {};
		\node [style=none] (12) at (2.25, -2.25) {};
		\node [style=none] (13) at (2.25, -1.25) {};
		\node [style=medium box] (14) at (2.25, -0.75) {$p_F$};
		\node [style=gray dot] (15) at (2.25, 0.75) {};
		\node [style=none] (16) at (2.25, 2) {};
		\node [style=none] (17) at (-1.25, -1) {};
		\node [style=none] (18) at (-2.75, -1) {};
		\node [style=none] (19) at (-2, 1) {};
		\node [style=white dot] (20) at (-2, 0) {};
	\end{pgfonlayer}
	\begin{pgfonlayer}{edgelayer}
		\draw (0.center) to (1.center);
		\draw [style=boldedge] (2.center) to (3.center);
		\draw [style=boldedge] (3.center) to (4.center);
		\draw [style=boldedge] (4.center) to (5.center);
		\draw [style=boldedge] (5.center) to (2.center);
		\draw (7.center) to (8.center);
		\draw (12.center) to (13.center);
		\draw [in=-150, out=90, looseness=1.00] (10.center) to (15);
		\draw [in=-30, out=90, looseness=1.00] (11.center) to (15);
		\draw (15) to (16.center);
		\draw [in=-150, out=90, looseness=1.00] (18.center) to (20);
		\draw [in=-30, out=90, looseness=1.00] (17.center) to (20);
		\draw (20) to (19.center);
	\end{pgfonlayer}
\end{tikzpicture} \qquad
\begin{tikzpicture}
	\begin{pgfonlayer}{nodelayer}
		\node [style=none] (0) at (-2, 1) {};
		\node [style=none] (1) at (-2, 2) {};
		\node [style=none] (2) at (-3.25, 1) {};
		\node [style=none] (3) at (-0.75, 1) {};
		\node [style=none] (4) at (-0.75, -1) {};
		\node [style=none] (5) at (-3.25, -1) {};
		\node [style=left label] (6) at (-3.5, 0.75) {$F$};
		\node [style=none] (7) at (-2, -2) {};
		\node [style=none] (8) at (-2, -1) {};
		\node [style=none] (9) at (0.25, 0) {$=$};
		\node [style=none] (10) at (1.75, -2) {};
		\node [style=none] (11) at (1.75, -1) {};
		\node [style=small box] (12) at (1.75, -0.5) {$u_F$};
		\node [style=gray dot] (13) at (1.75, 0.75) {};
		\node [style=none] (14) at (1.75, 2) {};
		\node [style=none] (15) at (-2, 1) {};
		\node [style=white dot] (16) at (-2, 0) {};
	\end{pgfonlayer}
	\begin{pgfonlayer}{edgelayer}
		\draw (0.center) to (1.center);
		\draw [style=boldedge] (2.center) to (3.center);
		\draw [style=boldedge] (3.center) to (4.center);
		\draw [style=boldedge] (4.center) to (5.center);
		\draw [style=boldedge] (5.center) to (2.center);
		\draw (7.center) to (8.center);
		\draw (10.center) to (11.center);
		\draw (13) to (14.center);
		\draw (16) to (15.center);
	\end{pgfonlayer}
\end{tikzpicture} \qquad
\begin{tikzpicture}
	\begin{pgfonlayer}{nodelayer}
		\node [style=none] (0) at (-2.25, -2) {};
		\node [style=none] (1) at (-2.25, -2.75) {};
		\node [style=none] (2) at (-3.5, -2) {};
		\node [style=none] (3) at (-1, -2) {};
		\node [style=none] (4) at (-1, 0) {};
		\node [style=none] (5) at (-3.5, 0) {};
		\node [style=left label] (6) at (-3.75, -0.25) {$F$};
		\node [style=none] (7) at (-2.25, 0.5) {};
		\node [style=none] (8) at (-2.25, 0) {};
		\node [style=none] (9) at (0.25, 0) {$=$};
		\node [style=none] (10) at (1, 1.25) {};
		\node [style=none] (11) at (2.5, 1.25) {};
		\node [style=medium box] (12) at (-2.25, 1) {$p_F$};
		\node [style=gray dot] (13) at (1.75, 0.25) {};
		\node [style=none] (14) at (1.75, -1) {};
		\node [style=none] (15) at (-1.5, 0) {};
		\node [style=none] (16) at (-3, 0) {};
		\node [style=none] (17) at (-2.25, -2) {};
		\node [style=white dot] (18) at (-2.25, -1) {};
		\node [style=none] (19) at (-3, 1.25) {};
		\node [style=none] (20) at (-3, 2.25) {};
		\node [style=none] (21) at (-1.5, 1.25) {};
		\node [style=none] (22) at (-1.5, 2.25) {};
	\end{pgfonlayer}
	\begin{pgfonlayer}{edgelayer}
		\draw (0.center) to (1.center);
		\draw [style=boldedge] (2.center) to (3.center);
		\draw [style=boldedge] (3.center) to (4.center);
		\draw [style=boldedge] (4.center) to (5.center);
		\draw [style=boldedge] (5.center) to (2.center);
		\draw (7.center) to (8.center);
		\draw [in=150, out=-90, looseness=1.00] (10.center) to (13);
		\draw [in=30, out=-90, looseness=1.00] (11.center) to (13);
		\draw (13) to (14.center);
		\draw [in=150, out=-90, looseness=1.00] (16.center) to (18);
		\draw [in=30, out=-90, looseness=1.00] (15.center) to (18);
		\draw (18) to (17.center);
		\draw (20.center) to (19.center);
		\draw (22.center) to (21.center);
	\end{pgfonlayer}
\end{tikzpicture} \qquad
\begin{tikzpicture}
	\begin{pgfonlayer}{nodelayer}
		\node [style=none] (0) at (-2, -1.5) {};
		\node [style=none] (1) at (-2, -2.5) {};
		\node [style=none] (2) at (-3.25, -1.5) {};
		\node [style=none] (3) at (-0.75, -1.5) {};
		\node [style=none] (4) at (-0.75, 0.5) {};
		\node [style=none] (5) at (-3.25, 0.5) {};
		\node [style=left label] (6) at (-3.5, 0.25) {$F$};
		\node [style=none] (7) at (-2, 1.25) {};
		\node [style=none] (8) at (-2, 0.5) {};
		\node [style=none] (9) at (0.25, 0) {$=$};
		\node [style=small box] (10) at (-2, 1.75) {$u_F$};
		\node [style=gray dot] (11) at (1.25, 0.5) {};
		\node [style=none] (12) at (1.25, -0.75) {};
		\node [style=none] (13) at (-2, -1.5) {};
		\node [style=white dot] (14) at (-2, -0.5) {};
	\end{pgfonlayer}
	\begin{pgfonlayer}{edgelayer}
		\draw (0.center) to (1.center);
		\draw [style=boldedge] (2.center) to (3.center);
		\draw [style=boldedge] (3.center) to (4.center);
		\draw [style=boldedge] (4.center) to (5.center);
		\draw [style=boldedge] (5.center) to (2.center);
		\draw (7.center) to (8.center);
		\draw (11) to (12.center);
		\draw (14) to (13.center);
	\end{pgfonlayer}
\end{tikzpicture}
\]

Dagger-hypergraph functors are important in particular because they define models of the free dagger-hypergraph category into a semantic category. We will make use of the following construction in building the particular models we use to show completeness in Section~\ref{sec:completeness}.

\begin{theorem}
  If $M$ and $N$ are (dagger-)hypergraph functors, then $(P,p_P,u_P)$ defined as:
  \[ P(A) := M(A) \otimes N(A) \qquad\qquad P(f) := M(f) \otimes N(f) \]
  \[
  \begin{tikzpicture}
	\begin{pgfonlayer}{nodelayer}
		\node [style=medium box] (0) at (0.75, 0) {$p_M$};
		\node [style=medium box] (1) at (3.25, 0) {$p_N$};
		\node [style=none] (2) at (0.75, -1.25) {};
		\node [style=none] (3) at (3.25, -1.25) {};
		\node [style=none] (4) at (0.75, -0.5) {};
		\node [style=none] (5) at (3.25, -0.5) {};
		\node [style=none] (6) at (0.25, 0.5) {};
		\node [style=none] (7) at (0.25, 1.5) {};
		\node [style=none] (8) at (2.75, 0.5) {};
		\node [style=none] (9) at (3.75, 0.5) {};
		\node [style=none] (10) at (3.75, 1.5) {};
		\node [style=none] (11) at (1.25, 0.5) {};
		\node [style=none] (12) at (2.75, 1.5) {};
		\node [style=none] (13) at (1.25, 1.5) {};
		\node [style=medium box] (14) at (-3.5, 0) {$p_P$};
		\node [style=none] (15) at (-4, 0.5) {};
		\node [style=none] (16) at (-3.5, -1.25) {};
		\node [style=none] (17) at (-4, 1.25) {};
		\node [style=none] (18) at (-3.5, -0.5) {};
		\node [style=none] (19) at (-3, 1.25) {};
		\node [style=none] (20) at (-3, 0.5) {};
		\node [style=none] (21) at (-1.5, 0) {$:=$};
	\end{pgfonlayer}
	\begin{pgfonlayer}{edgelayer}
		\draw (2.center) to (4.center);
		\draw (3.center) to (5.center);
		\draw (6.center) to (7.center);
		\draw (9.center) to (10.center);
		\draw [in=-90, out=90, looseness=1.00] (8.center) to (13.center);
		\draw [in=-90, out=90, looseness=1.00] (11.center) to (12.center);
		\draw (16.center) to (18.center);
		\draw (15.center) to (17.center);
		\draw (20.center) to (19.center);
	\end{pgfonlayer}
\end{tikzpicture}
  \qquad\qquad
  \begin{tikzpicture}
	\begin{pgfonlayer}{nodelayer}
		\node [style=small box] (0) at (0.25, 0) {$u_M$};
		\node [style=small box] (1) at (2, 0) {$u_N$};
		\node [style=none] (2) at (0.25, -1.25) {};
		\node [style=none] (3) at (2, -1.25) {};
		\node [style=none] (4) at (0.25, -0.5) {};
		\node [style=none] (5) at (2, -0.5) {};
		\node [style=small box] (6) at (-3.25, 0) {$u_P$};
		\node [style=none] (7) at (-3.25, -1.25) {};
		\node [style=none] (8) at (-3.25, -0.5) {};
		\node [style=none] (9) at (-1.5, 0) {$:=$};
	\end{pgfonlayer}
	\begin{pgfonlayer}{edgelayer}
		\draw (2.center) to (4.center);
		\draw (3.center) to (5.center);
		\draw (7.center) to (8.center);
	\end{pgfonlayer}
\end{tikzpicture}
  \]
  is also a (dagger-)hypergraph functor.
\end{theorem}

Let \MGCat be the 2-category of hypergraph categories, strong monoidal functors preserving the SCFA structure, and monoidal natural transformations. Let $\MGCat^\dagger$ be the 2-category of dagger-hypergraph categories, dagger-strong monoidal functors preserving the dSCFA structure, and monoidal natural transformations.

Note how we only assume a hypergraph category category is symmetric monoidal, rather than traced symmetric monoidal or compact closed. That is because this extra structure comes for free.

\begin{theorem}
  A (dagger-)hypergraph category is (dagger-)compact closed, with a coherent choice of self-dual compact structure for each object $A$, where by coherent we mean for all objects $A, B$, the following equations are satisfied:
  \[
  \begin{tikzpicture}
	\path [use as bounding box] (-5.25,-0.5) rectangle (5.25,0.75);
	\begin{pgfonlayer}{nodelayer}
		\node [style=none] (0) at (-3.5, -0.5) {};
		\node [style=none] (1) at (-1, -0.5) {};
		\node [style=left label] (2) at (-3.75, -0.25) {$A \otimes B$};
		\node [style=none] (3) at (0, 0) {$=$};
		\node [style=right label] (4) at (3.75, -0.25) {$A$};
		\node [style=none] (5) at (3.5, -0.5) {};
		\node [style=none] (6) at (1, -0.5) {};
		\node [style=right label] (7) at (4.75, -0.25) {$B$};
		\node [style=none] (8) at (4.5, -0.5) {};
		\node [style=none] (9) at (2, -0.5) {};
	\end{pgfonlayer}
	\begin{pgfonlayer}{edgelayer}
		\draw [in=90, out=90, looseness=1.75] (0.center) to (1.center);
		\draw [in=90, out=90, looseness=1.75] (5.center) to (6.center);
		\draw [in=90, out=90, looseness=1.75] (8.center) to (9.center);
	\end{pgfonlayer}
\end{tikzpicture} \qquad\qquad \begin{tikzpicture}
	\path [use as bounding box] (-3,-0.75) rectangle (3,1);
	\begin{pgfonlayer}{nodelayer}
		\node [style=none] (0) at (-1, 0) {};
		\node [style=none] (1) at (-2.75, 0) {};
		\node [style=none] (2) at (-2.75, -0.75) {};
		\node [style=none] (3) at (-1, -0.75) {};
		\node [style=none] (4) at (0, 0) {$=$};
		\node [style=none] (5) at (1, 0.25) {};
		\node [style=none] (6) at (1, -0.75) {};
		\node [style=none] (7) at (2.75, 0.25) {};
		\node [style=none] (8) at (2.75, -0.75) {};
	\end{pgfonlayer}
	\begin{pgfonlayer}{edgelayer}
		\draw [in=90, out=90, looseness=1.75] (0.center) to (1.center);
		\draw (2.center) to (1.center);
		\draw (3.center) to (0.center);
		\draw [in=90, out=90, looseness=1.75] (7.center) to (5.center);
		\draw [in=-90, out=90, looseness=1.00] (8.center) to (5.center);
		\draw [in=-90, out=90, looseness=1.00] (6.center) to (7.center);
	\end{pgfonlayer}
\end{tikzpicture}
  \]
\end{theorem}

\begin{proof}
  We define the compact structure on $A$ in terms of its SCFA.
  
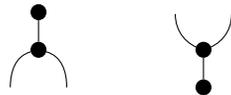
\[ \begin{tikzpicture}
	\begin{pgfonlayer}{nodelayer}
		\node [style=none] (0) at (-0.75, -1) {};
		\node [style=none] (1) at (0.75, -1) {};
		\node [style=black dot] (2) at (0, 0) {};
		\node [style=black dot] (3) at (0, 1) {};
	\end{pgfonlayer}
	\begin{pgfonlayer}{edgelayer}
		\draw [in=-165, out=90, looseness=1.00] (0.center) to (2);
		\draw [in=-15, out=90, looseness=1.00] (1.center) to (2);
		\draw (2) to (3);
	\end{pgfonlayer}
\end{tikzpicture} \qquad\qquad \begin{tikzpicture}
	\begin{pgfonlayer}{nodelayer}
		\node [style=none] (0) at (-0.75, 1) {};
		\node [style=none] (1) at (0.75, 1) {};
		\node [style=black dot] (2) at (0, 0) {};
		\node [style=black dot] (3) at (0, -1) {};
	\end{pgfonlayer}
	\begin{pgfonlayer}{edgelayer}
		\draw [in=-90, out=165, looseness=1.00] (2.center) to (0);
		\draw [in=-90, out=15, looseness=1.00] (2.center) to (1);
		\draw (3) to (2);
	\end{pgfonlayer}
\end{tikzpicture} \]
  We can use the Frobenius axioms to show this cap and cup satisfy the snake equations.
  
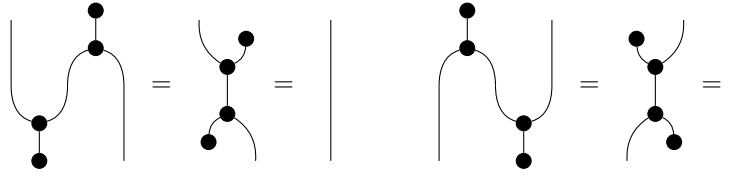
\[ \begin{tikzpicture}
	\begin{pgfonlayer}{nodelayer}
		\node [style=none] (0) at (-5.75, 0) {};
		\node [style=black dot] (1) at (-5, -1) {};
		\node [style=black dot] (2) at (-5, -2) {};
		\node [style=black dot] (3) at (-3.5, 1) {};
		\node [style=none] (4) at (-2.75, 0) {};
		\node [style=black dot] (5) at (-3.5, 2) {};
		\node [style=none] (6) at (-4.25, 0) {};
		\node [style=none] (7) at (-5.75, 1.75) {};
		\node [style=none] (8) at (-2.75, -2) {};
		\node [style=none] (9) at (-1.75, 0) {$=$};
		\node [style=black dot] (10) at (0, -0.75) {};
		\node [style=black dot] (11) at (0, 0.5) {};
		\node [style=black dot] (12) at (0.5, 1.25) {};
		\node [style=black dot] (13) at (-0.5, -1.5) {};
		\node [style=none] (14) at (0.75, -2) {};
		\node [style=none] (15) at (-0.75, 1.75) {};
		\node [style=none] (16) at (1.5, 0) {$=$};
		\node [style=none] (17) at (2.75, 1.75) {};
		\node [style=none] (18) at (2.75, -2) {};
	\end{pgfonlayer}
	\begin{pgfonlayer}{edgelayer}
		\draw (2) to (1);
		\draw [in=90, out=-15, looseness=1.00] (3) to (4.center);
		\draw [in=90, out=-165, looseness=1.00] (3) to (6.center);
		\draw (3) to (5);
		\draw (8.center) to (4.center);
		\draw (0.center) to (7.center);
		\draw [in=-90, out=15, looseness=1.00] (1) to (6.center);
		\draw [in=-90, out=165, looseness=1.00] (1) to (0.center);
		\draw [bend right, looseness=1.00] (14.center) to (10);
		\draw [bend left, looseness=1.00] (13) to (10);
		\draw (10) to (11);
		\draw [bend right, looseness=1.00] (11) to (12);
		\draw [bend left, looseness=1.00] (11) to (15.center);
		\draw (18.center) to (17.center);
	\end{pgfonlayer}
\end{tikzpicture} \qquad\qquad \begin{tikzpicture}
	\begin{pgfonlayer}{nodelayer}
		\node [style=none] (0) at (-2.75, 0) {};
		\node [style=black dot] (1) at (-3.5, -1) {};
		\node [style=black dot] (2) at (-3.5, -2) {};
		\node [style=black dot] (3) at (-5, 1) {};
		\node [style=none] (4) at (-5.75, 0) {};
		\node [style=black dot] (5) at (-5, 2) {};
		\node [style=none] (6) at (-4.25, 0) {};
		\node [style=none] (7) at (-2.75, 1.75) {};
		\node [style=none] (8) at (-5.75, -2) {};
		\node [style=none] (9) at (-1.75, 0) {$=$};
		\node [style=black dot] (10) at (0, -0.75) {};
		\node [style=black dot] (11) at (0, 0.5) {};
		\node [style=black dot] (12) at (-0.5, 1.25) {};
		\node [style=black dot] (13) at (0.5, -1.5) {};
		\node [style=none] (14) at (-0.75, -2) {};
		\node [style=none] (15) at (0.75, 1.75) {};
		\node [style=none] (16) at (1.5, 0) {$=$};
		\node [style=none] (17) at (2.75, 1.75) {};
		\node [style=none] (18) at (2.75, -2) {};
	\end{pgfonlayer}
	\begin{pgfonlayer}{edgelayer}
		\draw (2) to (1);
		\draw [in=90, out=-165, looseness=1.00] (3) to (4.center);
		\draw [in=90, out=-15, looseness=1.00] (3) to (6.center);
		\draw (3) to (5);
		\draw (8.center) to (4.center);
		\draw (0.center) to (7.center);
		\draw [in=-90, out=165, looseness=1.00] (1) to (6.center);
		\draw [in=-90, out=15, looseness=1.00] (1) to (0.center);
		\draw [bend left, looseness=1.00] (14.center) to (10);
		\draw [bend right, looseness=1.00] (13) to (10);
		\draw (10) to (11);
		\draw [bend left, looseness=1.00] (11) to (12);
		\draw [bend right, looseness=1.00] (11) to (15.center);
		\draw (18.center) to (17.center);
	\end{pgfonlayer}
\end{tikzpicture} \]
  The coherence equations follow from (co)commutativity of the Frobenius algebras and the definition of $A \otimes B$ as:
  \[ (A, \whitemult) \otimes (B, \graymult) \ =\ (A \otimes B, \prodmult{white dot}{gray dot}) \]
\end{proof}

\begin{remark}
  This notion is very close to that of `compact closed with a coherent self-duality', as introduced by Selinger in~\cite{SelingerSelfDual}. When a category $\mathcal C$ has a monoidal product that is free on objects, the two notions coincide. We simply let the (non-self) dual of $A = B_1 \otimes \ldots \otimes B_n$ be the `reversed' object $A^* := B_n \otimes \ldots \otimes B_1$ and define caps, cups, and the self-duality $A \cong A^*$ in the obvious way.

  If $\mathcal C$ is \textit{not} free on objects, we can pass by a standard construction to a new category $\widetilde{\mathcal{C}}$ that is free on objects then define the requisite structure.
\end{remark}

\section{Dot-diagrams}\label{sec:dot-diagrams}

We now define dot-diagrams, following the construction from~\cite{HasegawaHofmannPlotkin,SelingerCompleteness}.

\begin{definition}
  A (dagger) dot-diagram $F = (B^F, D^F, I^F, O^F, \Lbox^F, \Ldot^F, \Win^F, \Wout^F)$ for a (dagger) signature $\Sigma$ consists of the following data:
  \begin{itemize}
    \item A finite set $B^F$ of \textit{boxes},
    \item a finite set $D^F$ of multi-edges, or \textit{dots},
    \item a finite set $I^F := \{ \inp_1, \ldots, \inp_m \}$ of \textit{diagram inputs},
    \item a finite set $O^F := \{ \outp_1, \ldots, \outp_n \}$ of \textit{diagram outputs},
    \item labelling functions $\Lbox^F : B^F \to \Mor$ and $\Ldot^F : D^F \to \Obj$,
    \item and wiring functions $\Win : \Inputs^F \to D^F$ and $\Wout : \Outputs^F \to D^F$, where:
    \begin{align*}
      \Inputs^F & := \{ (b,i) \ |\ 
        b \in B^F,\,
        n = |\dom(\Lbox^F(b))|,\,
        1 \leq i \leq n \} + I^F \\
      \Outputs^F & := \{ (j,b) \ |\ 
        b \in B^F,\,
        m = |\cod(\Lbox^F(b))|,\,
        1 \leq j \leq m \} + O^F
    \end{align*}
  \end{itemize}
  Satisfying the conditions:
  \begin{itemize}
    \item $\Ldot^F(\Win^F(b,i)) = \dom(\Lbox^F(b))[i]$
    \item $\Ldot^F(\Wout^F(j,b)) = \cod(\Lbox^F(b))[j]$
  \end{itemize}
\end{definition}

Dot-diagrams are much like string diagrams, except that rather than requiring each wire to be associated with precisely one input and output (either to a box or the diagram as a while), we allow a single `wire', which we now called a \textit{dot}, to be connected to many inputs/outputs. 
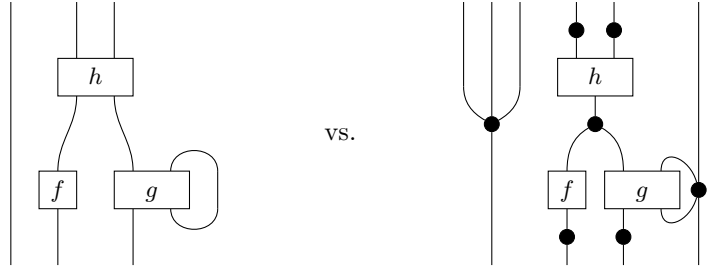
\[
\begin{tikzpicture}
	\begin{pgfonlayer}{nodelayer}
		\node [style=small box] (0) at (-0.75, -1.5) {$f$};
		\node [style=medium box] (1) at (0.25, 1.5) {$h$};
		\node [style=medium box] (2) at (1.75, -1.5) {$g$};
		\node [style=none] (3) at (1.25, -3.5) {};
		\node [style=none] (4) at (1.25, -2) {};
		\node [style=none] (5) at (-2, 3.5) {};
		\node [style=none] (6) at (-2, -3.5) {};
		\node [style=none] (7) at (2.25, -1) {};
		\node [style=none] (8) at (3.5, -1) {};
		\node [style=none] (9) at (-0.75, -1) {};
		\node [style=none] (10) at (-0.25, 1) {};
		\node [style=none] (11) at (-0.25, 2) {};
		\node [style=none] (12) at (-0.25, 3.5) {};
		\node [style=none] (13) at (1.25, -1) {};
		\node [style=none] (14) at (0.75, 1) {};
		\node [style=none] (15) at (-0.75, -3.5) {};
		\node [style=none] (16) at (-0.75, -2) {};
		\node [style=none] (17) at (2.25, -2) {};
		\node [style=none] (18) at (3.5, -2) {};
		\node [style=none] (19) at (0.75, 2) {};
		\node [style=none] (20) at (0.75, 3.5) {};
	\end{pgfonlayer}
	\begin{pgfonlayer}{edgelayer}
		\draw (3.center) to (4.center);
		\draw (6.center) to (5.center);
		\draw [in=90, out=90, looseness=1.50] (7.center) to (8.center);
		\draw [in=-90, out=90, looseness=1.00] (9.center) to (10.center);
		\draw (11.center) to (12.center);
		\draw [in=-90, out=90, looseness=1.00] (13.center) to (14.center);
		\draw (15.center) to (16.center);
		\draw [in=-90, out=-90, looseness=1.50] (17.center) to (18.center);
		\draw (8.center) to (18.center);
		\draw (19.center) to (20.center);
	\end{pgfonlayer}
\end{tikzpicture}
\qquad\qquad \textrm{vs.} \qquad\qquad
\begin{tikzpicture}
	\begin{pgfonlayer}{nodelayer}
		\node [style=small box] (0) at (-0.75, -1.5) {$f$};
		\node [style=medium box] (1) at (0, 1.5) {$h$};
		\node [style=medium box] (2) at (1.25, -1.5) {$g$};
		\node [style=none] (3) at (0.75, -3.5) {};
		\node [style=black dot] (4) at (0.75, -2.75) {};
		\node [style=none] (5) at (0.75, -2) {};
		\node [style=none] (6) at (0.75, -1) {};
		\node [style=black dot] (7) at (0, 0.25) {};
		\node [style=none] (8) at (0, 1) {};
		\node [style=none] (9) at (2.75, -3.5) {};
		\node [style=black dot] (10) at (2.75, -1.5) {};
		\node [style=none] (11) at (2.75, 3.5) {};
		\node [style=none] (12) at (1.75, -1) {};
		\node [style=none] (13) at (1.75, -2) {};
		\node [style=none] (14) at (-0.5, 3.5) {};
		\node [style=none] (15) at (-0.5, 2) {};
		\node [style=black dot] (16) at (-0.5, 2.75) {};
		\node [style=none] (17) at (-0.75, -1) {};
		\node [style=none] (18) at (-0.75, -2) {};
		\node [style=none] (19) at (-0.75, -3.5) {};
		\node [style=black dot] (20) at (-0.75, -2.75) {};
		\node [style=black dot] (21) at (-2.75, 0.25) {};
		\node [style=none] (22) at (-2.75, -3.5) {};
		\node [style=none] (23) at (-3.5, 1.25) {};
		\node [style=none] (24) at (-3.5, 3.5) {};
		\node [style=none] (25) at (-2, 1.25) {};
		\node [style=none] (26) at (-2, 3.5) {};
		\node [style=none] (27) at (-2.75, 3.5) {};
		\node [style=none] (28) at (0.5, 3.5) {};
		\node [style=none] (29) at (0.5, 2) {};
		\node [style=black dot] (30) at (0.5, 2.75) {};
	\end{pgfonlayer}
	\begin{pgfonlayer}{edgelayer}
		\draw (3.center) to (4);
		\draw (4) to (5.center);
		\draw [in=-30, out=90, looseness=1.00] (6.center) to (7);
		\draw (7) to (8.center);
		\draw (9.center) to (10);
		\draw (10) to (11.center);
		\draw [in=105, out=90, looseness=1.75] (12.center) to (10);
		\draw [in=-90, out=-105, looseness=1.75] (10) to (13.center);
		\draw (15.center) to (16);
		\draw (16) to (14.center);
		\draw [in=-150, out=90, looseness=1.00] (17.center) to (7);
		\draw (19.center) to (20);
		\draw (20) to (18.center);
		\draw (22.center) to (21);
		\draw [in=-90, out=150, looseness=1.00] (21) to (23.center);
		\draw [in=-90, out=30, looseness=1.00] (21) to (25.center);
		\draw (25.center) to (26.center);
		\draw (23.center) to (24.center);
		\draw (21) to (27.center);
		\draw (29.center) to (30);
		\draw (30) to (28.center);
	\end{pgfonlayer}
\end{tikzpicture}
\]
Thus, dots serve as multi-edges in dot-diagrams. Normal string diagrams can be seen as a subset of dot-diagrams, where we write dots with one input and one output just as wires:
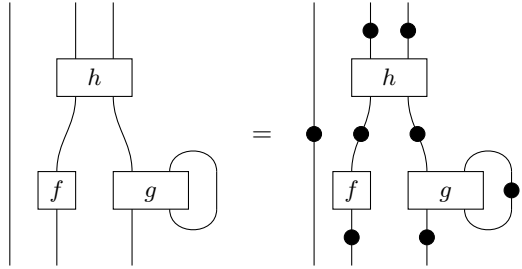
\[
\begin{tikzpicture}
	\begin{pgfonlayer}{nodelayer}
		\node [style=small box] (0) at (-0.75, -1.5) {$f$};
		\node [style=medium box] (1) at (0.25, 1.5) {$h$};
		\node [style=medium box] (2) at (1.75, -1.5) {$g$};
		\node [style=none] (3) at (1.25, -3.5) {};
		\node [style=none] (4) at (1.25, -2) {};
		\node [style=none] (5) at (-2, 3.5) {};
		\node [style=none] (6) at (-2, -3.5) {};
		\node [style=none] (7) at (2.25, -1) {};
		\node [style=none] (8) at (3.5, -1) {};
		\node [style=none] (9) at (-0.75, -1) {};
		\node [style=none] (10) at (-0.25, 1) {};
		\node [style=none] (11) at (-0.25, 2) {};
		\node [style=none] (12) at (-0.25, 3.5) {};
		\node [style=none] (13) at (1.25, -1) {};
		\node [style=none] (14) at (0.75, 1) {};
		\node [style=none] (15) at (-0.75, -3.5) {};
		\node [style=none] (16) at (-0.75, -2) {};
		\node [style=none] (17) at (2.25, -2) {};
		\node [style=none] (18) at (3.5, -2) {};
		\node [style=none] (19) at (0.75, 2) {};
		\node [style=none] (20) at (0.75, 3.5) {};
	\end{pgfonlayer}
	\begin{pgfonlayer}{edgelayer}
		\draw (3.center) to (4.center);
		\draw (6.center) to (5.center);
		\draw [in=90, out=90, looseness=1.50] (7.center) to (8.center);
		\draw [in=-90, out=90, looseness=1.00] (9.center) to (10.center);
		\draw (11.center) to (12.center);
		\draw [in=-90, out=90, looseness=1.00] (13.center) to (14.center);
		\draw (15.center) to (16.center);
		\draw [in=-90, out=-90, looseness=1.50] (17.center) to (18.center);
		\draw (8.center) to (18.center);
		\draw (19.center) to (20.center);
	\end{pgfonlayer}
\end{tikzpicture} \quad = \quad
\begin{tikzpicture}
	\begin{pgfonlayer}{nodelayer}
		\node [style=small box] (0) at (-0.75, -1.5) {$f$};
		\node [style=medium box] (1) at (0.25, 1.5) {$h$};
		\node [style=medium box] (2) at (1.75, -1.5) {$g$};
		\node [style=none] (3) at (1.25, -3.5) {};
		\node [style=none] (4) at (1.25, -2) {};
		\node [style=none] (5) at (-1.75, 3.5) {};
		\node [style=none] (6) at (-1.75, -3.5) {};
		\node [style=none] (7) at (2.25, -1) {};
		\node [style=none] (8) at (3.5, -1) {};
		\node [style=none] (9) at (-0.75, -1) {};
		\node [style=none] (10) at (-0.25, 1) {};
		\node [style=none] (11) at (-0.25, 2) {};
		\node [style=none] (12) at (-0.25, 3.5) {};
		\node [style=none] (13) at (1.25, -1) {};
		\node [style=none] (14) at (0.75, 1) {};
		\node [style=none] (15) at (-0.75, -3.5) {};
		\node [style=none] (16) at (-0.75, -2) {};
		\node [style=none] (17) at (2.25, -2) {};
		\node [style=none] (18) at (3.5, -2) {};
		\node [style=none] (19) at (0.75, 2) {};
		\node [style=none] (20) at (0.75, 3.5) {};
		\node [style=black dot] (21) at (-1.75, 0) {};
		\node [style=black dot] (22) at (-0.5, 0) {};
		\node [style=black dot] (23) at (1, 0) {};
		\node [style=black dot] (24) at (3.5, -1.5) {};
		\node [style=black dot] (25) at (1.25, -2.75) {};
		\node [style=black dot] (26) at (-0.75, -2.75) {};
		\node [style=black dot] (27) at (-0.25, 2.75) {};
		\node [style=black dot] (28) at (0.75, 2.75) {};
	\end{pgfonlayer}
	\begin{pgfonlayer}{edgelayer}
		\draw (3.center) to (4.center);
		\draw (6.center) to (5.center);
		\draw [in=90, out=90, looseness=1.50] (7.center) to (8.center);
		\draw [in=-90, out=90, looseness=1.00] (9.center) to (10.center);
		\draw (11.center) to (12.center);
		\draw [in=-90, out=90, looseness=1.00] (13.center) to (14.center);
		\draw (15.center) to (16.center);
		\draw [in=-90, out=-90, looseness=1.50] (17.center) to (18.center);
		\draw (8.center) to (18.center);
		\draw (19.center) to (20.center);
	\end{pgfonlayer}
\end{tikzpicture}
\]

\begin{definition}
  Let $F$ and $G$ be dot-diagrams where $I^F = I^G$ and $O^F = O^G$. A \textit{dot-diagram homomorphism} $\Phi : F \to G$ is a pair of functions $(\Phi_b, \Phi_d)$ that respect all of the structure:
  \begin{align*}
    \Lbox^G(\Phi_b(b)) & = \Lbox^F(b) &
    \Win^G(\Phi_b(b), i) & = \Phi_d(\Win^F(b,i)) &
    \Win^F(\inp_i) & = \Win^G(\inp_i) \\
    \Ldot^G(\Phi_d(d)) & = \Ldot^F(d) &
    \Wout^G(j, \Phi_b(b)) & = \Phi_d(\Wout^F(j, b)) &
    \Wout^F(\outp_j) & = \Wout^G(\outp_j)
  \end{align*}
\end{definition}


It will simplify to proof to first restrict to the case where diagrams have no (global) inputs/outputs, and no `free-floating' dots. We call these simple closed dot-diagrams.

\begin{definition}
  A (dagger) dot-diagram is called \textit{simple} when the wiring functions $\Win, \Wout$ are both \textbf{surjections}, and it is called \textit{closed} when $I^M = O^M = \{ \}$.
\end{definition}

The only significant difference with the definition from~\cite{SelingerCompleteness} is that the bijections $\Win$, $\Wout$ are replaced with surjections. This makes the `dots' in dot-diagrams serve as multi-edges, rather than single wires.

When considering homomorphisms of simple, closed dot-diagrams, three of the equations above become redundant:
\[
\Ldot^G(\Phi_d(d)) = \Ldot^F(d) \qquad\qquad
\Win^F(\inp_i) = \Win^G(\inp_i) \qquad\qquad
\Wout^F(\outp_j) = \Wout^G(\outp_j)
\]
The first is forced by surjectivity and the fact that connections between boxes and dots must respect $\Sigma$. The other two are vacuously satisfied for closed diagrams. Furthermore, we only need to require the box function $\Psi_b$ to be surjective in order to obtain a surjective homomorphism of dot-diagrams.

\begin{theorem}\label{thm:surjection-lift}
  Let $(\Psi_b, \Psi_d)$ be a homomorphism of dot-diagrams, and let $\Psi_b$ be a surjective function. Then $\Psi_d$ is also surjective.
\end{theorem}

\begin{proof}
  Since $\Psi_b$ is a surjection, it indices a surjection $\widehat\Psi_b : \textrm{Inputs}^F \to \textrm{Inputs}^G$. Then, by the homomorphism conditions, the following diagram commutes:
  \[
  \begin{tikzpicture}
  \matrix(m)[cdiag]{
    \textrm{Inputs}^F & B^F \\
    \textrm{Inputs}^G & B^G \\
  };
  \path [arrs]
    (m-1-1) edge [->>] node {$\Win^F$} (m-1-2)
    (m-2-1) edge [->>] node {$\Win^G$} (m-2-2)
    (m-1-1) edge [->>] node [swap] {$\widehat\Psi_b$} (m-2-1)
    (m-1-2) edge [->] node {$\Psi_d$} (m-2-2);
  \end{tikzpicture}
  \]
  Suppose $\Psi_d$ is not surjective, then $\Psi_d \circ \Win^F$ ($= \Win^G \circ \Psi_b$) would not be surjective, which is a contradiction. Thus $\Psi_d$ is surjective.
\end{proof}

Unlike in~\cite{SelingerCompleteness}, it is the property of \textit{surjectivity} that lifts from the box function to the whole homomorphism, not \textit{isomorphism}. In fact, there are examples of homomorphisms with bijective box functions whose dot function is merely a surjection, and \textit{not} a bijection, e.g.
\ctikzfig{bij-surj-hm}

\begin{definition}
  For a (dagger) monoidal signature $\Sigma$, the (dagger) hypergraph category $\DotCat(\Sigma)$ of dot-diagrams is defined as follows:
  \begin{itemize}
    \item \textbf{Objects} are words in $\Obj^*$,
    \item \textbf{Morphisms} are (isomorphism classes of) dot-diagrams $F$ such that for all $i,j$:
    \[ \Ldot(\Win(\inp_i)) = \dom(F)[i] \qquad\qquad
       \Ldot(\Wout(\outp_j)) = \cod(F)[j] \]
    where $\dom(F)[i]$ is the $i$-th object in the input word of $F$ and $\cod(F)[j]$ the $j$-th object in the output word.
    \item \textbf{Composition} is defined by pushing out over adjacent dots:
    \[
    \begin{tikzpicture}[scale=2]
    \matrix(m)[cdiag]{
      O^F & I^G & D^G \\
      D^F &     & D^{G\circ F} \\
    };
    \path [arrs]
    (m-1-1) edge node {$\cong$} (m-1-2)
    (m-1-2) edge node {$\Win^G$} (m-1-3)
    (m-1-1) edge node [swap] {$\Wout^F$} (m-2-1)
    (m-1-3) edge node {$\iota^G$} (m-2-3)
    (m-2-1) edge node [swap] {$\iota^F$} (m-2-3);
    \NWbracket{(m-2-3)}
    \end{tikzpicture}
    \]
    where $\Win^G$ and $\Wout^F$ are (restrictions of) the wiring functions of $G$ and $F$. Then:
    \[ G \circ F := (B^F + B^G, D^{G \circ F}, I^F, O^G, \Lbox^{G\circ F}, \Ldot^{G\circ F}, \Win^{G\circ F}, \Wout^{G\circ F}) \]
    for $\Lbox^{G\circ F} := [\Lbox^F,\Lbox^G]$ induced by the coproduct of boxes, $\Ldot^{G\circ F}$ by the pushout of dots:
    \[
    \begin{tikzpicture}[scale=2]
    \matrix(m)[cdiag]{
      O^F & I^G & D^G & \\
      D^F &     & D^{G\circ F} &     \\
          &     &              & \Mor \\
    };
    \path [arrs]
    (m-1-1) edge node {$\cong$} (m-1-2)
    (m-1-2) edge node {$\Win^G$} (m-1-3)
    (m-1-1) edge node [swap] {$\Wout^F$} (m-2-1)
    (m-1-3) edge node {$\iota^G$} (m-2-3)
    (m-2-1) edge node [swap] {$\iota^F$} (m-2-3)
    (m-1-3) edge [bend left=30] node {$\Ldot^G$} (m-3-4)
    (m-2-1) edge [bend right=15] node {$\Ldot^F$} (m-3-4)
    (m-2-3) edge [dashed] node {$\Ldot^{G\circ F}$} (m-3-4);
    \NWbracket{(m-2-3)}
    \end{tikzpicture}
    \]
    and the wiring functions defined by:
    \[
    \begin{tikzpicture}[scale=2]
    \matrix(m)[cdiag]{
      \Inputs^{G \circ F} & \Inputs^F + \Inputs^G & & D^F + D^G & D^{G \circ F} \\
    };
    \path [arrs]
    (m-1-1) edge [right hook-latex] node {} (m-1-2)
    (m-1-2) edge node {$\Win^F + \Win^G$} (m-1-4)
    (m-1-4) edge [>=latex,->>] node {} (m-1-5);
    \end{tikzpicture}
    \]
    \[
    \begin{tikzpicture}[scale=2]
    \matrix(m)[cdiag]{
      \Outputs^{G \circ F} & \Outputs^F + \Outputs^G & & D^F + D^G & D^{G \circ F} \\
    };
    \path [arrs]
    (m-1-1) edge [right hook-latex] node {} (m-1-2)
    (m-1-2) edge node {$\Wout^F + \Wout^G$} (m-1-4)
    (m-1-4) edge [>=latex,->>] node {} (m-1-5);
    \end{tikzpicture}
    \]
    \item the \textbf{monoidal product} is defined as the disjoint union of two dot-diagrams:
    \[ F \otimes G := (B^F + B^G,\, D^F + D^G,\, I^F + I^G,\, O^F + O^G,\, [\Lbox^F, \Lbox^G],\, [\Ldot^F, \Ldot^G],\, \Win^F + \Win^G,\, \Wout^F + \Wout^G) \]
    where $I^F + I^G$ and $O^F + O^G$ are chosen coproducts of the form:
    \begin{align*}
      \{ \inp_1, \ldots, \inp_m \} + \{ \inp_1, \ldots, \inp_{m'} \}
      & := \{ \inp_1, \ldots, \inp_m, \inp_{m+1}, \ldots, \inp_{m+m'} \} \\
      \{ \outp_1, \ldots, \outp_n \} + \{ \outp_1, \ldots, \outp_{n'} \}
      & := \{ \outp_1, \ldots, \outp_n, \outp_{n+1}, \ldots, \outp_{n+n'} \}
    \end{align*}
    The \textbf{monoidal unit} is given by the empty dot-diagram.
    \item \textbf{Swap} maps are defined as pairs of dots with swapped outputs (or inputs):
    \begin{align*}
      \gamma_{A,B} & := (\{ d_1,d_2 \}, \{ \}, \{ \inp_1, \inp_2 \}, \{ \outp_1, \outp_2 \}, \{ \}, \{ d_1 \mapsto A, d_2 \mapsto B \}, \\
            & \qquad \{ \inp_1 \mapsto d_1, \inp_2 \mapsto d_2 \}, \{ \outp_1 \mapsto d_2, \outp_2 \mapsto d_1 \})
    \end{align*}
    \item For a fixed object $A$, we define the \textbf{SCFA structure} on $A$. Each map is the unique dot-diagram with a single dot of type $A$ and the appropriate number of inputs/outputs:
    \begin{align*}
      \mu_A & := (\{ d \}, \{ \}, \{ \inp_1, \inp_2 \}, \{ \outp_1 \}, \{ \},
        \{ d \mapsto A \}, !, !) \\
      \eta_A & := (\{ d \}, \{ \}, \{ \}, \{ \outp_1 \}, \{ \},
        \{ d \mapsto A \}, !, !) \\
      \delta_A & := (\{ d \}, \{ \}, \{ \inp_1 \}, \{ \outp_1, \outp_2 \}, \{ \},
        \{ d \mapsto A \}, !, !) \\
      \epsilon_A & := (\{ d \}, \{ \}, \{ \inp_1 \}, \{ \}, \{ \},
        \{ d \mapsto A \}, !, !) \\
    \end{align*}
    where $!$ is in all cases the terminal map into $\{ d \}$. \textbf{Identities} arise in the same way, taking a single input and output:
    \[ 1_A := (\{ d \}, \{ \}, \{ \inp_1 \}, \{ \outp_1 \}, \{ \},
        \{ d \mapsto A \}, !, !) \]
    \item If $\Sigma$ is a dagger signature, the \textbf{dagger structure} is obtained by changing all of the boxes to their daggered versions and interchanging the role of inputs/outputs:
    \[ F^\dagger = (B^F, D^F, O^F, I^F, \dagger \circ \Lbox^F, \Ldot^F, \Wout^F, \Win^F) \]
    where we change an element `$\inp_j$' to `$\outp_j$' as appropriate, and vice-versa.
  \end{itemize}
\end{definition}

This is very close to the combinatoric string diagram presentation of the free traced symmetric monoidal category given in~\cite{HasegawaHofmannPlotkin}. The biggest departure is in composition. In order to obtain a new dot-diagram as a composition of dot-diagrams, adjacent dots fuse together:
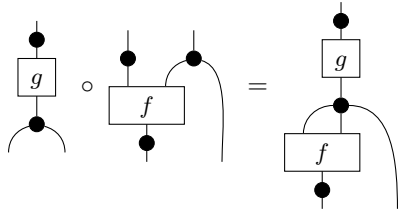
\[ \begin{tikzpicture}
	\begin{pgfonlayer}{nodelayer}
		\node [style=none] (0) at (0, 0.75) {};
		\node [style=none] (1) at (0, 1.75) {};
		\node [style=black dot] (2) at (0, 1.25) {};
		\node [style=none] (3) at (0.75, -1.75) {};
		\node [style=none] (4) at (0, -0.25) {};
		\node [style=black dot] (5) at (0, -1) {};
		\node [style=none] (6) at (-0.75, -1.75) {};
		\node [style=small box] (7) at (0, 0.25) {$g$};
	\end{pgfonlayer}
	\begin{pgfonlayer}{edgelayer}
		\draw [in=-15, out=90, looseness=1.00] (3.center) to (5);
		\draw [in=-165, out=90, looseness=1.00] (6.center) to (5);
		\draw (0.center) to (1.center);
		\draw (5) to (4.center);
	\end{pgfonlayer}
\end{tikzpicture} \ \circ\ \begin{tikzpicture}
	\begin{pgfonlayer}{nodelayer}
		\node [style=medium box] (0) at (-0.75, -0.5) {$f$};
		\node [style=black dot] (1) at (0.5, 0.75) {};
		\node [style=none] (2) at (1.25, -2) {};
		\node [style=none] (3) at (-0.75, -2) {};
		\node [style=none] (4) at (-0.75, -1) {};
		\node [style=none] (5) at (-0.25, 0) {};
		\node [style=black dot] (6) at (-0.75, -1.5) {};
		\node [style=none] (7) at (-1.25, 0) {};
		\node [style=none] (8) at (-1.25, 1.5) {};
		\node [style=black dot] (9) at (-1.25, 0.75) {};
		\node [style=none] (10) at (0.5, 1.5) {};
	\end{pgfonlayer}
	\begin{pgfonlayer}{edgelayer}
		\draw (3.center) to (4.center);
		\draw [in=-15, out=90, looseness=0.75] (2.center) to (1);
		\draw [in=-165, out=90, looseness=1.00] (5.center) to (1);
		\draw (7.center) to (8.center);
		\draw (1) to (10.center);
	\end{pgfonlayer}
\end{tikzpicture} \ =\ \begin{tikzpicture}
	\begin{pgfonlayer}{nodelayer}
		\node [style=medium box] (0) at (-0.75, -1.75) {$f$};
		\node [style=black dot] (1) at (-0.25, -0.5) {};
		\node [style=none] (2) at (1.25, -3.25) {};
		\node [style=none] (3) at (-0.75, -3.25) {};
		\node [style=none] (4) at (-0.75, -2.25) {};
		\node [style=none] (5) at (-1.25, -1.25) {};
		\node [style=black dot] (6) at (-0.75, -2.75) {};
		\node [style=none] (7) at (-0.25, 0.25) {};
		\node [style=none] (8) at (-0.25, 1.25) {};
		\node [style=black dot] (9) at (-0.25, 1.75) {};
		\node [style=none] (10) at (-0.25, 2.25) {};
		\node [style=small box] (11) at (-0.25, 0.75) {$g$};
		\node [style=none] (12) at (-0.25, -1.25) {};
	\end{pgfonlayer}
	\begin{pgfonlayer}{edgelayer}
		\draw (3.center) to (4.center);
		\draw [in=0, out=90, looseness=1.00] (2.center) to (1);
		\draw [in=180, out=90, looseness=1.00] (5.center) to (1);
		\draw (1) to (7.center);
		\draw (8.center) to (10.center);
		\draw (12.center) to (1);
	\end{pgfonlayer}
\end{tikzpicture} \]

\begin{theorem}
  The category $\DotCat(\Sigma)$ is the free (dagger-)hypergraph category over a signature $\Sigma$. In other words, there is an equivalence of categories:
  \[ \textbf{Mod}(\Sigma, \mathcal C) \simeq \textbf{MGCat}(\DotCat(\Sigma), \mathcal C) \]
  or for $\Sigma$ a dagger-signature and $\mathcal C$ a dagger-hypergraph category:
  \[ \textbf{Mod}(\Sigma, \mathcal C) \simeq \textbf{MGCat}^\dagger(\DotCat(\Sigma), \mathcal C) \]
\end{theorem}

\begin{proof}
  Since the free category is characterised up to equivalence, let us show an equivalence (actually an isomorphism) between $\DotCat(\Sigma)$ and a more `obvious' representation of the free category. Let $\Sigma'$ be a signature containing $\Sigma$ and additional maps $(\mult, \unit, \comult, \counit)$ for each $A$ in the objects of $\Sigma$. Let $\Free(\Sigma')$ be the free traced symmetric monoidal category of $\Sigma'$ and let $\Free(\Sigma')_\equiv$ be the same, but with morphisms taken modulo the SCFA equations. This becomes a dagger-hypergraph category in the obvious way, and by construction, forms the \textit{free} dagger-hypergraph category.

  Since $\Free(\Sigma')$ is the free traced SMC, we can take its morphisms to be string diagrams. Define a functor $F : \Free(\Sigma')_\equiv \to \DotCat(\Sigma)$ that is identity-on-objects. For morphisms, it sends each box in $\Sigma$ to itself, each connected component of morphisms in $\Sigma' \backslash \Sigma$ to a dot, and each `blank' wire (i.e. a wire not otherwise touching a morphism in $\Sigma' \backslash \Sigma$) to a dot with one input and one output.
  \[ F \ ::\ 
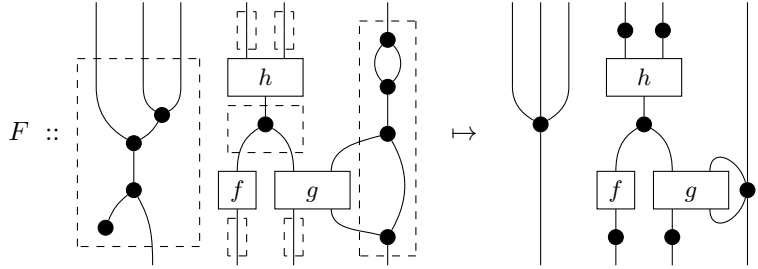
\begin{tikzpicture}
	\begin{pgfonlayer}{nodelayer}
		\node [style=small box] (0) at (-0.75, -1.5) {$f$};
		\node [style=medium box] (1) at (0, 1.5) {$h$};
		\node [style=medium box] (2) at (1.25, -1.5) {$g$};
		\node [style=none] (3) at (0.75, -3.5) {};
		\node [style=none] (4) at (0.75, -2.75) {};
		\node [style=none] (5) at (0.75, -2) {};
		\node [style=none] (6) at (0.75, -1) {};
		\node [style=black dot] (7) at (0, 0.25) {};
		\node [style=none] (8) at (0, 1) {};
		\node [style=none] (9) at (3.25, -3.5) {};
		\node [style=none] (10) at (3.25, 3.5) {};
		\node [style=none] (11) at (1.75, -1) {};
		\node [style=none] (12) at (1.75, -2) {};
		\node [style=none] (13) at (-0.5, 3.5) {};
		\node [style=none] (14) at (-0.5, 2) {};
		\node [style=none] (15) at (-0.5, 2.75) {};
		\node [style=none] (16) at (-0.75, -1) {};
		\node [style=none] (17) at (-0.75, -2) {};
		\node [style=none] (18) at (-0.75, -3.5) {};
		\node [style=none] (19) at (-0.75, -2.75) {};
		\node [style=none] (20) at (-3, -3.5) {};
		\node [style=none] (21) at (-4.5, 1.25) {};
		\node [style=none] (22) at (-4.5, 3.5) {};
		\node [style=none] (23) at (-2.25, 1.25) {};
		\node [style=none] (24) at (-2.25, 3.5) {};
		\node [style=none] (25) at (-3.25, 3.5) {};
		\node [style=none] (26) at (0.5, 3.5) {};
		\node [style=none] (27) at (0.5, 2) {};
		\node [style=none] (28) at (0.5, 2.75) {};
		\node [style=black dot] (29) at (-2.75, 0.5) {};
		\node [style=black dot] (30) at (-3.5, -0.25) {};
		\node [style=black dot] (31) at (-3.5, -1.5) {};
		\node [style=none] (32) at (-3.25, 1.25) {};
		\node [style=black dot] (33) at (-4.25, -2.5) {};
		\node [style=black dot] (34) at (3.25, -2.75) {};
		\node [style=black dot] (35) at (3.25, 0) {};
		\node [style=black dot] (36) at (3.25, 1.25) {};
		\node [style=black dot] (37) at (3.25, 2.5) {};
		\node [style=none] (38) at (-5, 2) {};
		\node [style=none] (39) at (-1.75, 2) {};
		\node [style=none] (40) at (-1.75, -3) {};
		\node [style=none] (41) at (-5, -3) {};
		\node [style=none] (42) at (4, 3) {};
		\node [style=none] (43) at (4, -3.25) {};
		\node [style=none] (44) at (2.5, 3) {};
		\node [style=none] (45) at (2.5, -3.25) {};
		\node [style=none] (46) at (-0.25, 3.25) {};
		\node [style=none] (47) at (-0.25, 2.25) {};
		\node [style=none] (48) at (-0.75, 3.25) {};
		\node [style=none] (49) at (-0.75, 2.25) {};
		\node [style=none] (50) at (0.25, 3.25) {};
		\node [style=none] (51) at (0.75, 2.25) {};
		\node [style=none] (52) at (0.75, 3.25) {};
		\node [style=none] (53) at (0.25, 2.25) {};
		\node [style=none] (54) at (-1, -2.25) {};
		\node [style=none] (55) at (-0.5, -3.25) {};
		\node [style=none] (56) at (-0.5, -2.25) {};
		\node [style=none] (57) at (-1, -3.25) {};
		\node [style=none] (58) at (0.5, -2.25) {};
		\node [style=none] (59) at (1, -3.25) {};
		\node [style=none] (60) at (1, -2.25) {};
		\node [style=none] (61) at (0.5, -3.25) {};
		\node [style=none] (62) at (-1, 0.75) {};
		\node [style=none] (63) at (1, -0.5) {};
		\node [style=none] (64) at (1, 0.75) {};
		\node [style=none] (65) at (-1, -0.5) {};
	\end{pgfonlayer}
	\begin{pgfonlayer}{edgelayer}
		\draw (3.center) to (4.center);
		\draw (4.center) to (5.center);
		\draw [in=-30, out=90, looseness=1.00] (6.center) to (7);
		\draw (7) to (8.center);
		\draw (14.center) to (15.center);
		\draw (15.center) to (13.center);
		\draw [in=-150, out=90, looseness=1.00] (16.center) to (7);
		\draw (18.center) to (19.center);
		\draw (19.center) to (17.center);
		\draw (23.center) to (24.center);
		\draw (21.center) to (22.center);
		\draw (27.center) to (28.center);
		\draw (28.center) to (26.center);
		\draw (32.center) to (25.center);
		\draw [in=-90, out=30, looseness=1.00] (29) to (23.center);
		\draw [in=-90, out=150, looseness=1.00] (29) to (32.center);
		\draw [bend right=15, looseness=1.00] (30) to (29);
		\draw [in=-90, out=150, looseness=1.00] (30) to (21.center);
		\draw (31) to (30);
		\draw [in=-60, out=90, looseness=1.00] (20.center) to (31);
		\draw [bend left=15, looseness=1.00] (33) to (31);
		\draw (9.center) to (34);
		\draw [in=-90, out=165, looseness=1.00] (34) to (12.center);
		\draw [bend right, looseness=1.00] (34) to (35);
		\draw [in=-165, out=90, looseness=1.00] (11.center) to (35);
		\draw (35) to (36);
		\draw [bend right=45, looseness=1.00] (36) to (37);
		\draw [bend left=45, looseness=1.00] (36) to (37);
		\draw (37) to (10.center);
		\draw [style=dashed edge] (38.center) to (39.center);
		\draw [style=dashed edge] (39.center) to (40.center);
		\draw [style=dashed edge] (40.center) to (41.center);
		\draw [style=dashed edge] (41.center) to (38.center);
		\draw [style=dashed edge] (44.center) to (42.center);
		\draw [style=dashed edge] (42.center) to (43.center);
		\draw [style=dashed edge] (43.center) to (45.center);
		\draw [style=dashed edge] (45.center) to (44.center);
		\draw [style=dashed edge] (48.center) to (46.center);
		\draw [style=dashed edge] (46.center) to (47.center);
		\draw [style=dashed edge] (47.center) to (49.center);
		\draw [style=dashed edge] (49.center) to (48.center);
		\draw [style=dashed edge] (50.center) to (52.center);
		\draw [style=dashed edge] (52.center) to (51.center);
		\draw [style=dashed edge] (51.center) to (53.center);
		\draw [style=dashed edge] (53.center) to (50.center);
		\draw [style=dashed edge] (54.center) to (56.center);
		\draw [style=dashed edge] (56.center) to (55.center);
		\draw [style=dashed edge] (55.center) to (57.center);
		\draw [style=dashed edge] (57.center) to (54.center);
		\draw [style=dashed edge] (58.center) to (60.center);
		\draw [style=dashed edge] (60.center) to (59.center);
		\draw [style=dashed edge] (59.center) to (61.center);
		\draw [style=dashed edge] (61.center) to (58.center);
		\draw [style=dashed edge] (62.center) to (64.center);
		\draw [style=dashed edge] (64.center) to (63.center);
		\draw [style=dashed edge] (63.center) to (65.center);
		\draw [style=dashed edge] (65.center) to (62.center);
	\end{pgfonlayer}
\end{tikzpicture} \quad \mapsto \quad \begin{tikzpicture}
	\begin{pgfonlayer}{nodelayer}
		\node [style=small box] (0) at (-0.75, -1.5) {$f$};
		\node [style=medium box] (1) at (0, 1.5) {$h$};
		\node [style=medium box] (2) at (1.25, -1.5) {$g$};
		\node [style=none] (3) at (0.75, -3.5) {};
		\node [style=black dot] (4) at (0.75, -2.75) {};
		\node [style=none] (5) at (0.75, -2) {};
		\node [style=none] (6) at (0.75, -1) {};
		\node [style=black dot] (7) at (0, 0.25) {};
		\node [style=none] (8) at (0, 1) {};
		\node [style=none] (9) at (2.75, -3.5) {};
		\node [style=black dot] (10) at (2.75, -1.5) {};
		\node [style=none] (11) at (2.75, 3.5) {};
		\node [style=none] (12) at (1.75, -1) {};
		\node [style=none] (13) at (1.75, -2) {};
		\node [style=none] (14) at (-0.5, 3.5) {};
		\node [style=none] (15) at (-0.5, 2) {};
		\node [style=black dot] (16) at (-0.5, 2.75) {};
		\node [style=none] (17) at (-0.75, -1) {};
		\node [style=none] (18) at (-0.75, -2) {};
		\node [style=none] (19) at (-0.75, -3.5) {};
		\node [style=black dot] (20) at (-0.75, -2.75) {};
		\node [style=black dot] (21) at (-2.75, 0.25) {};
		\node [style=none] (22) at (-2.75, -3.5) {};
		\node [style=none] (23) at (-3.5, 1.25) {};
		\node [style=none] (24) at (-3.5, 3.5) {};
		\node [style=none] (25) at (-2, 1.25) {};
		\node [style=none] (26) at (-2, 3.5) {};
		\node [style=none] (27) at (-2.75, 3.5) {};
		\node [style=none] (28) at (0.5, 3.5) {};
		\node [style=none] (29) at (0.5, 2) {};
		\node [style=black dot] (30) at (0.5, 2.75) {};
	\end{pgfonlayer}
	\begin{pgfonlayer}{edgelayer}
		\draw (3.center) to (4);
		\draw (4) to (5.center);
		\draw [in=-30, out=90, looseness=1.00] (6.center) to (7);
		\draw (7) to (8.center);
		\draw (9.center) to (10);
		\draw (10) to (11.center);
		\draw [in=105, out=90, looseness=1.75] (12.center) to (10);
		\draw [in=-90, out=-105, looseness=1.75] (10) to (13.center);
		\draw (15.center) to (16);
		\draw (16) to (14.center);
		\draw [in=-150, out=90, looseness=1.00] (17.center) to (7);
		\draw (19.center) to (20);
		\draw (20) to (18.center);
		\draw (22.center) to (21);
		\draw [in=-90, out=150, looseness=1.00] (21) to (23.center);
		\draw [in=-90, out=30, looseness=1.00] (21) to (25.center);
		\draw (25.center) to (26.center);
		\draw (23.center) to (24.center);
		\draw (21) to (27.center);
		\draw (29.center) to (30);
		\draw (30) to (28.center);
	\end{pgfonlayer}
\end{tikzpicture} \]
  By Theorem~\ref{thm:spider}, any two connected diagrams built from $(\mult, \unit, \comult, \counit)$ are equal, so this mapping does not depend on the choice of representative in $\Free(\Sigma')_\equiv$. We can show this preserves composition by again invoking Theorem~\ref{thm:spider}, and it preserves the rest of the dagger-hypergraph structure by definition of the respective categories. In the other direction, define a functor $G$ which is also identity on objects and sends a dot-diagram to a string diagram with each dot replaced by a canonical tree/cotree $S_m^n$ (as in Definition~\ref{def:spider}) built from the SCFA structure.
  \[ G \ ::\ 
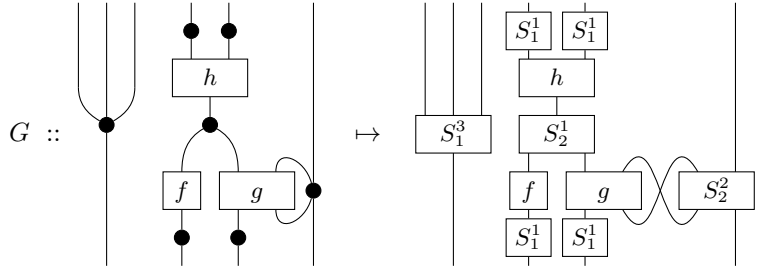
\begin{tikzpicture}
	\begin{pgfonlayer}{nodelayer}
		\node [style=small box] (0) at (-0.75, -1.5) {$f$};
		\node [style=medium box] (1) at (0, 1.5) {$h$};
		\node [style=medium box] (2) at (1.25, -1.5) {$g$};
		\node [style=none] (3) at (0.75, -3.5) {};
		\node [style=black dot] (4) at (0.75, -2.75) {};
		\node [style=none] (5) at (0.75, -2) {};
		\node [style=none] (6) at (0.75, -1) {};
		\node [style=black dot] (7) at (0, 0.25) {};
		\node [style=none] (8) at (0, 1) {};
		\node [style=none] (9) at (2.75, -3.5) {};
		\node [style=black dot] (10) at (2.75, -1.5) {};
		\node [style=none] (11) at (2.75, 3.5) {};
		\node [style=none] (12) at (1.75, -1) {};
		\node [style=none] (13) at (1.75, -2) {};
		\node [style=none] (14) at (-0.5, 3.5) {};
		\node [style=none] (15) at (-0.5, 2) {};
		\node [style=black dot] (16) at (-0.5, 2.75) {};
		\node [style=none] (17) at (-0.75, -1) {};
		\node [style=none] (18) at (-0.75, -2) {};
		\node [style=none] (19) at (-0.75, -3.5) {};
		\node [style=black dot] (20) at (-0.75, -2.75) {};
		\node [style=black dot] (21) at (-2.75, 0.25) {};
		\node [style=none] (22) at (-2.75, -3.5) {};
		\node [style=none] (23) at (-3.5, 1.25) {};
		\node [style=none] (24) at (-3.5, 3.5) {};
		\node [style=none] (25) at (-2, 1.25) {};
		\node [style=none] (26) at (-2, 3.5) {};
		\node [style=none] (27) at (-2.75, 3.5) {};
		\node [style=none] (28) at (0.5, 3.5) {};
		\node [style=none] (29) at (0.5, 2) {};
		\node [style=black dot] (30) at (0.5, 2.75) {};
	\end{pgfonlayer}
	\begin{pgfonlayer}{edgelayer}
		\draw (3.center) to (4);
		\draw (4) to (5.center);
		\draw [in=-30, out=90, looseness=1.00] (6.center) to (7);
		\draw (7) to (8.center);
		\draw (9.center) to (10);
		\draw (10) to (11.center);
		\draw [in=105, out=90, looseness=1.75] (12.center) to (10);
		\draw [in=-90, out=-105, looseness=1.75] (10) to (13.center);
		\draw (15.center) to (16);
		\draw (16) to (14.center);
		\draw [in=-150, out=90, looseness=1.00] (17.center) to (7);
		\draw (19.center) to (20);
		\draw (20) to (18.center);
		\draw (22.center) to (21);
		\draw [in=-90, out=150, looseness=1.00] (21) to (23.center);
		\draw [in=-90, out=30, looseness=1.00] (21) to (25.center);
		\draw (25.center) to (26.center);
		\draw (23.center) to (24.center);
		\draw (21) to (27.center);
		\draw (29.center) to (30);
		\draw (30) to (28.center);
	\end{pgfonlayer}
\end{tikzpicture} \quad \mapsto \quad \begin{tikzpicture}
	\begin{pgfonlayer}{nodelayer}
		\node [style=small box] (0) at (-0.75, -1.5) {$f$};
		\node [style=medium box] (1) at (0, 1.5) {$h$};
		\node [style=medium box] (2) at (1.25, -1.5) {$g$};
		\node [style=none] (3) at (0.75, -3.5) {};
		\node [style=small box] (4) at (0.75, -2.75) {$S_1^1$};
		\node [style=none] (5) at (0.75, -2) {};
		\node [style=none] (6) at (0.75, -1) {};
		\node [style=none] (7) at (4.75, -3.5) {};
		\node [style=none] (8) at (4.75, 3.5) {};
		\node [style=none] (9) at (1.75, -1) {};
		\node [style=none] (10) at (-0.75, 3.5) {};
		\node [style=none] (11) at (-0.75, 2) {};
		\node [style=small box] (12) at (-0.75, 2.75) {$S_1^1$};
		\node [style=none] (13) at (-0.75, -2) {};
		\node [style=none] (14) at (-0.75, -3.5) {};
		\node [style=small box] (15) at (-0.75, -2.75) {$S_1^1$};
		\node [style=none] (16) at (-2.75, -3.5) {};
		\node [style=none] (17) at (-3.5, 0.5) {};
		\node [style=none] (18) at (-3.5, 3.5) {};
		\node [style=none] (19) at (-2, 0.5) {};
		\node [style=none] (20) at (-2, 3.5) {};
		\node [style=none] (21) at (-2.75, 3.5) {};
		\node [style=none] (22) at (0.75, 3.5) {};
		\node [style=none] (23) at (0.75, 2) {};
		\node [style=small box] (24) at (0.75, 2.75) {$S_1^1$};
		\node [style=medium box] (25) at (-2.75, 0) {$S_1^3$};
		\node [style=none] (26) at (-2.75, 0.5) {};
		\node [style=none] (27) at (-2.75, -0.5) {};
		\node [style=medium box] (28) at (4.25, -1.5) {$S_2^2$};
		\node [style=none] (29) at (4.75, -1) {};
		\node [style=none] (30) at (4.75, -2) {};
		\node [style=none] (31) at (3.75, -2) {};
		\node [style=none] (32) at (3.75, -1) {};
		\node [style=none] (33) at (1.75, -2) {};
		\node [style=none] (34) at (0.75, -0.5) {};
		\node [style=medium box] (35) at (0, 0) {$S_2^1$};
		\node [style=none] (36) at (-0.75, -1) {};
		\node [style=none] (37) at (-0.75, -0.5) {};
		\node [style=none] (38) at (0, 0.5) {};
		\node [style=none] (39) at (0, 1) {};
	\end{pgfonlayer}
	\begin{pgfonlayer}{edgelayer}
		\draw (3.center) to (4);
		\draw (4) to (5.center);
		\draw (11.center) to (12);
		\draw (12) to (10.center);
		\draw (14.center) to (15);
		\draw (15) to (13.center);
		\draw (19.center) to (20.center);
		\draw (17.center) to (18.center);
		\draw (23.center) to (24);
		\draw (24) to (22.center);
		\draw (26.center) to (21.center);
		\draw (16.center) to (27.center);
		\draw (7.center) to (30.center);
		\draw (29.center) to (8.center);
		\draw [in=-120, out=60, looseness=2.25] (9.center) to (31.center);
		\draw [in=-60, out=120, looseness=2.25] (32.center) to (33.center);
		\draw (6.center) to (34.center);
		\draw (36.center) to (37.center);
		\draw (38.center) to (39.center);
	\end{pgfonlayer}
\end{tikzpicture} \]
  Again we can proof dagger-hypergraph functoriality by invoking Theorem~\ref{thm:spider}. The round-trip $F \circ G$ on morphism in $\DotCat(\Sigma)$ sends a dot-diagram to the identical dot-diagram, so $F \circ G = 1_{\DotCat(\Sigma)}$. The other round-trip $G \circ F$ sends a string diagram to another string diagram, where all connected components of $\Sigma' \backslash \Sigma$ are sent to a canonical form w.r.t. to the SCFA axioms. Clearly this is in the same $\equiv$-equivalence class as the original morphism, so $G \circ F = 1_{\Free(\Sigma')_\equiv}$.
\end{proof}

\section{The category $\Mat(R)$}\label{sec:matrix-category}

While it is more typical to define a category of matrices whose objects are natural numbers and whose matrices are indexed by sets of the form $\{ 1, \ldots, m \}$, it will be more convenient for our purposes to use an equivalent category with arbitrary finite sets as indices. It will also be convenient for the proofs to define $\Mat(R)$ for all unital semirings, not just fields. Thus, to fix notation, we will now define this version of $\Mat(R)$ and give its (dagger-)hypergraph structure.

Let $R$ be a unital semiring with a (possibly trivial) involution operation $\overline{(-)}$. Let $\Mat(R)$ be the category whose objects are finite sets ${\mathcal I},{\mathcal J},\ldots$ and whose morphisms $\psi : {\mathcal I} \to {\mathcal J}$ are $|{\mathcal J}| \times |{\mathcal I}|$ matrices, i.e. matrices whose rows are indexed by $i \in {\mathcal I}$ and whose columns are indexed by $j \in {\mathcal J}$, with composition and identities defined as usual.
\[
(g \circ f)_i^k := \sum_k f_i^j g_j^k
\qquad\qquad
(1_{\mathcal I})_i^j := \begin{cases}
1 & \textrm{ if } i = j \\
0 & \textrm{ otherwise}
\end{cases}
\]
Where the juxtaposition $f_i^k g_k^l$ means multiplication in $R$. Note we have adopted tensor notation where inputs/columns appear as lower indices and outputs/rows appear as upper indices:
\[
\psi = \left(\begin{matrix}
  \psi_0^0 & \psi_1^0 \\
  \psi_0^1 & \psi_1^1
\end{matrix}\right)
\]
We define a monidal product on objects as ${\mathcal I} \otimes {\mathcal J} := {\mathcal I} \times {\mathcal J}$ and on morphisms as the \textit{Kronecker product} of matrices:
\[
  (\psi \otimes \phi)_{(i,j)}^{(k,l)} = \psi_i^k \phi_j^l
\]
Note that we typically drop brackets and commas, when there can be no confusion:
\[ \psi_{i_1...i_m}^{j_1...j_n} \]
$\Mat(K)$ is a dagger-monoidal category, letting:
\[ (\psi^\dagger)_i^j := \overline{(\psi_j^i)} \]
The hypergraph structure is given by `generalised Kronecker delta' matrices.
\begin{align*}
  \mu_{ij}^k & = \begin{cases}
    1 & \textrm{ if } i = j = k \\
    0 & \textrm{ otherwise}
  \end{cases} &
  \eta^i & = 1 \\
  \delta_{i}^{jk} & = \begin{cases}
    1 & \textrm{ if } i = j = k \\
    0 & \textrm{ otherwise}
  \end{cases} &
  \epsilon_i & = 1
\end{align*}
Clearly any connected diagram of these matrices just becomes a bigger Kronecker delta, with the general case being:
\[
  (S_m^n)_{i_1...i_m}^{j_1...j_n} = \begin{cases}
    1 & \textrm{ if for some $k$, and all $\alpha$:\  } i_\alpha = j_\alpha = k \\
    0 & \textrm{ otherwise}
  \end{cases}
\]
From this and the observation that $S_1^1$ is the identity matrix, all the SCFA identities follow.

It is possible to characterise functors out of the free hypergraph category in terms of matrix (i.e. tensor) contraction. For a dot diagram $F$ and a hypergraph functor $M : \DotCat(\Sigma) \to \Mat(R)$, let
\[ \textrm{Idx}(F,M) := \prod_{d\in D^F} M(\Ldot^F(d)) \]
be the set of \textit{indexing functions} for $F$. The elements of $\textrm{Idx}(F,M)$ can be seen as tuples, but for out purposes, it will be more convenient to write them using (dependent) function notation. That is, each $\phi \in \textrm{Idx}(F,M)$ assigns an element $\phi(d) \in M(\Ldot^F(d))$ to each dot in $F$.

\begin{theorem}\label{thm:contraction}
  Let $F : I \to I$ be a morphism in $\DotCat(\Sigma)$ represented by a closed, simple dot-diagram (up to isomorphism), and let $M$ be a hypergraph functor. Then:
  \[
  M(F) =
  \sum_{\phi \in \textrm{Idx}(F,M)} \prod_{b \in B^F}
  M(\Lbox^F(b))_{\phi(\Win^F(b,1)) , ... , \phi(\Win^F(b,m))}^{\phi(\Wout^F(1,b)), ... , \phi(\Wout^F(n,b))}
  \]
  where $\sum$ and $\prod$ are sums and products in $K$, respectively.
\end{theorem}

\begin{proof}
  The only difference between this interpretation and the usual interpretation of a string diagram as a tensor contraction is that a single index can be repeated any number of times, not just as a single pair of upper and lower indices.

  Rather that summing over dots, we could have instead summed over individual wires. Since no wire is connected to more than one box, it is uniquely identified by being the $i$-th input or $j$-th output of box $b$. Each of these wires is then connected to a map of the form $S_m^n$.
  \[
  \sum_{(b,1)...(b,m_b),(1,b)...(n_b,b),...} \left(\prod_{b \in B^F}
  M(\Lbox^F(b))_{(b,1)...(b,m_b)}^{(1,b)...(n_b,b)} \right)
  \left(\prod_{d \in D^F}
  (S_{m_d}^{n_d})_{\Wout^{-1}(d)}^{\Win^{-1}(d)} \right)
  \]
  By commutativity, the order of the indices on the $S$-maps doesn't matter, so we've abused notation by writing them as the appropriate sets. Since the maps $S_m^n$ are generalised Kronecker deltas, we can simplify this expression by removing redundant indices e.g.
  \[
  \sum_{ijk} f_{ij}^{k} g_{k}  (S_{2}^{1})_{ij}^{k} = 
  \sum_{i} f_{ii}^{i} g_{i}
  \]
  So we are left with one distinct index corresponding to each dot, and the indices which used to be labelled by wires are now labelled according to the dot each wire was connected to. This is precisely the form stated in the theorem.
\end{proof}

\section{Completeness of hypergraph categories}\label{sec:completeness}

In this section, we will prove the completeness theorem for hypergraph categories, without the dagger structure. In the next section, we will show the dagger case by tweaking the proof a little bit. Throughout this and the next section, let $K$ be a field of characteristic $0$. We begin by proving the simple, closed case, then generalising via corollaries.


\begin{theorem}\label{thm:hypergraph-completeness}
  Two simple, closed dot-diagrams $F$, and $G$ are isomorphic iff for all hypergraph functors $\llbracket-\rrbracket : \DotCat(\Sigma) \to \Mat(K)$, $\llbracket F \rrbracket = \llbracket G \rrbracket$.
\end{theorem}

\begin{proof}
  Suppose for simple, closed dot-diagrams $F$ and $G$ that $F \not\cong G$. Then, we need to show that there exists some functor $\llbracket-\rrbracket : \DotCat(\Sigma) \to \Mat(K)$ such that $\llbracket F \rrbracket \neq \llbracket G \rrbracket$. Suppose firstly that they have a different number of dots. Then, we can distinguish them by the `dot-counting' functor. This sends every object in $\Sigma$ to a two-element set, and every morphism to the following matrix:
  \[ (\llbracket f \rrbracket^d)_{x_1,...x_m}^{y_1,...,y_n} = 1 \]
  Then, using the form of evaluation given by Theorem~\ref{thm:contraction}, we can compute:
  \[
  \llbracket F \rrbracket^d \ =
  \sum_{\phi \in \textrm{Idx}(F, \llbracket \rrbracket^d)}
   \prod_{b \in B^F}
   (\llbracket \Lbox(b)^F \rrbracket^d)_{...}^{...} \ =
  \sum_{\phi \in \textrm{Idx}(F, \llbracket \rrbracket^d)} 1 \ =\ 
  |\textrm{Idx}(F, \llbracket \rrbracket^d)| \ =\  2^{|D^F|}
  \]
  If $F$ and $G$ have different numbers of dots, $\llbracket F\rrbracket^d \neq \llbracket G \rrbracket^d$. Therefore, assume $F$ and $G$ have the same number of dots. We will now construct a functor $\llbracket - \rrbracket_{FG}$ that distinguishes them. We do this in two phases, first let $\mathbf{X} := \{ X_b \ |\  b \in B^F \}$ be a set of variables indexed by the boxes in $F$ and let $\ZX$ be a polynomial ring. Then, we will construct a functor:
  \[ \llbracket - \rrbracket_{F} : \DotCat(\Sigma) \to \Mat(\ZX) \]
  On objects, let $\llbracket A \rrbracket_F$ be the (finite) set of all wires in $F$ labelled by $A$:
  \[ \llbracket A \rrbracket_F := \Ldot^{-1}(A) \]
  For each box $b$ labelled by $f$, let:
  \[ (f_b)_{x_1,...,x_m}^{y_1,...y_n} := \begin{cases}
      X_b & \textrm{ if }
        x_i = \Win^F(b,i) \textrm{ and }
        y_j = \Wout^F(j,b) \\
      0 & \textrm{ otherwise}
  \end{cases} \]
  The interpretation itself is then defined as a sum over all of the boxes in $F$ labelled $f$:
  \[ \llbracket f \rrbracket_F :=
        \sum_{b \in (\Lbox^F)^{-1}(f)} f_b
  \]
  We can then compute $\llbracket G \rrbracket_F$ by writing it in the form given by Theorem~\ref{thm:contraction}.
  \begin{align*}
    \llbracket G \rrbracket_F & =
    \sum_{\phi \in \textrm{Idx}(G,\llbracket\rrbracket_F)}
    \prod_{b \in B^G}
    (\llbracket\Lbox^G(b)\rrbracket_F)_{\phi(\Win^G(b,1)) , ... , \phi(\Win^G(b,m))}^{\phi(\Wout^G(1,b)), ... , \phi(\Wout^G(n,b))} \\
    & =
    \sum_{\phi \in \textrm{Idx}(G,\llbracket\rrbracket_F)}
    \prod_{b \in B^G}
    \sum_{b' \in (\Lbox^F)^{-1}(\Lbox^G(b))}
    (f_{b'})_{\phi(\Win^G(b,1)) , ... , \phi(\Win^G(b,m))}^{\phi(\Wout^G(1,b)), ... , \phi(\Wout^G(n,b))} \\
    & =
    \sum_{\phi \in \textrm{Idx}(G,\llbracket\rrbracket_F)}
    \prod_{b \in B^G}
    \sum_{b' \in B^G}
    \begin{cases}
    X_{b'}
      & \textrm{if } \Lbox^F(b') = \Lbox^G(b) \textrm{ and } \\
      & \phi(\Win^G(i,b)) = \Win^F(i,b') \textrm{ and } \\
      & \phi(\Wout^G(b,j)) = \Wout^F(b',j) \textrm{ for all } i, j \\
    0 & \textrm{otherwise}
    \end{cases}
  \end{align*}
  Applying distributivity yields:
  \begin{equation}\label{eq:big-sum}
    \llbracket G \rrbracket_F =
    \sum_{\phi \in \textrm{Idx}(G,\llbracket\rrbracket_F)}
    \sum_{\psi : B^G \to B^F}
    \prod_{b \in B^G}
    \begin{cases}
    X_{\psi(b)}
      & \textrm{if } \Lbox^F(\psi(b)) = \Lbox^G(b) \textrm{ and } \\
      & \phi(\Win^G(i,b)) = \Win^F(i,\psi(b)) \textrm{ and } \\
      & \phi(\Wout^G(b,j)) = \Wout^F(\psi(b),j) \textrm{ for all } i, j \\
    0 & \textrm{otherwise}
    \end{cases}
  \end{equation}
  We will refer to the coefficient of $\llbracket G \rrbracket_F$ corresponding of $\prod_b X_b$ as the `magic coefficient'. Clearly if two dot-diagrams $G, G'$ have different magic coefficients, then the polynomials $\llbracket G \rrbracket_F$ and $\llbracket G' \rrbracket_F$ will not be equal.

  Why do we call it the magic coefficient? Since the summation in~\eqref{eq:big-sum} is over a pair of functions $\phi, \psi$, computing the value of the magic coefficient amounts to identifying for which $\psi,\phi$ the following conditions are satisfied exactly once for each box $b' \in B^F$:
  \[ \psi(b) = b' \qquad
     \Lbox^F(\psi(b)) = \Lbox^G(b) \qquad
     \phi(\Wout^G(b,j)) = \Wout^F(\psi(b),j)
  \]
  The latter two conditions give precisely what it means for $(\psi,\phi)$ to be a homomorphism of dot-diagrams. As the product ranges over $b \in B^G$, the fact that these are satisfied exactly once means that $\psi$ is a bijection of boxes. Let $\textrm{BBij}(G,F)$ be the set of homomorphisms $(\psi,\phi)$ from $G$ to $F$ such that the box function is a bijection. It follows that the magic coefficient of $\llbracket G \rrbracket_F$ is $|\textrm{BBij}(G,F)|$.

  In particular, $\psi$ is a surjection, so by Theorem~\ref{thm:surjection-lift}, so too is $\phi$. But, since $F$ and $G$ have the same number of dots, a surjection from $G$ to $F$ is actually an isomorphism of dot-diagrams. Since we assumed $F \not\cong G$, the magic coefficient of $\llbracket G \rrbracket_F$ must be $0$, whereas the magic coefficient of $\llbracket F \rrbracket_F$ is $|\textrm{Aut}(F)|$. From this, we can conclude
  \[ \llbracket F \rrbracket_F \neq \llbracket G \rrbracket_F \]

  It only remains to turn this into a functor into $\Mat(K)$, rather than $\Mat(\ZX)$. Any ring homomorphism $h : \ZX \to K$ extends in the obvious way to a functor from $\Mat(\ZX)$ to $\Mat(K)$. In particular, any choice of values for the $X_b$ induces an evaluation functor $\textrm{ev} : \Mat(\ZX)\to \Mat(K)$. This defines a functor in $\Mat(K)$ by:
  \[ \llbracket - \rrbracket_{FG} := \textrm{ev}(\llbracket - \rrbracket_{F}) \]
  Since the polynomial $p := \llbracket F \rrbracket_F - \llbracket G \rrbracket_F$ is non-zero, it has finitely many roots in $K$. Since $K$ is infinite, choose values for $X_b$ in $K$ such that $p$ is non-zero. Then $\llbracket F \rrbracket_{FG} - \llbracket G \rrbracket_{FG} \neq 0$, and hence $\llbracket F \rrbracket_{FG} \neq \llbracket G \rrbracket_{FG}$.
\end{proof}

Extending to all morphisms in $\DotCat(\Sigma)$ is now straightforward. First, we eliminate the `simple' condition.

\begin{corollary}\label{cor:non-simple}
  Two closed morphisms $F$ and $G$ in $\DotCat(\Sigma)$ are equal iff for all hypergraph functors $\llbracket-\rrbracket : \DotCat(\Sigma) \to \Mat(K)$, $\llbracket F \rrbracket = \llbracket G \rrbracket$.
\end{corollary}

\begin{proof}
  Suppose $F$ and $G$ are closed, but not simple. The following functor counts the `free dots' of each type in $F$ and $G$, i.e. those not connected to a box. Fix a set of distinct prime numbers $\{ p_A \}$ for each object $A$ labelling a dot in $F$ or $G$.
  \[
  (\llbracket f \rrbracket_{FG}^f)_{i_1...i_m}^{j_1...j_n} = 
  \begin{cases}
    1 & \textrm{ if } i_k = j_l = 1 \\
    0 & \textrm{ otherwise}
  \end{cases}
  \]
  This sends every simple closed diagram to $1$ and closed diagram consisting of a single free dot of type $A$ to the number $p_A$. Let $F = F' \otimes d_F$ and $G = G' \otimes d_G$, where $F', G'$ are simple, and $d_F, d_G$ consist only of free dots. Then:
  \[ \llbracket F \rrbracket_{FG}^f =
     \llbracket F' \rrbracket_{FG}^f \llbracket d_F \rrbracket_{FG}^f = 
     \llbracket d_F \rrbracket_{FG}^f = (p_A)^{N_A} (p_B)^{N_B} (p_C)^{N_C} \ldots
  \]
  where $N_A$ is the number of free dots of type $A$. If $d_F \not\cong d_G$, then
  \[ \llbracket F \rrbracket_{FG}^f = \llbracket d_F \rrbracket_{FG}^f \neq \llbracket d_G \rrbracket_{FG}^f = \llbracket G \rrbracket_{FG}^f \]
  Otherwise, $d_F \cong d_G$. Let $\llbracket - \rrbracket_{FG}$ be the functor defined in the proof of Theorem~\ref{thm:hypergraph-completeness}. Then, it is easy to check that:
  \[ 0 \neq \llbracket d_F \rrbracket_{F'G'} = \llbracket d_G \rrbracket_{F'} \]
  from which follows that $\llbracket F \rrbracket_{F'G'} \neq \llbracket G \rrbracket_{F'G'}$.
\end{proof}

\begin{corollary}\label{cor:non-closed}
  Two morphisms $F$ and $G$ in $\DotCat(\Sigma)$ are equal iff for all hypergraph functors $\llbracket-\rrbracket : \DotCat(\Sigma) \to \Mat(K)$, $\llbracket F \rrbracket = \llbracket G \rrbracket$.
\end{corollary}

\begin{proof}
  Let $F, G : A \to B$ be two morphisms in $\DotCat(\Sigma)$. Let $\Sigma'$ be a new signature obtained by adding boxes $i_A : I \to A$ and $o_B : B \to I$. Then, there is a functor $E : \DotCat(\Sigma) \to \DotCat(\Sigma')$ defined in the obvious way.
  
  Then, as closed dot-diagrams, $o_B \circ E(F) \circ i_A \not\cong o_B \circ E(F) \circ i_A$. By the above arguments, there exists a functor $\llbracket - \rrbracket$ such that:
  \[ \llbracket o_B \circ E(F) \circ i_A \rrbracket \neq
     \llbracket o_B \circ E(G) \circ i_A \rrbracket \]
  It follows by functoriality that $\llbracket E(F) \rrbracket \neq \llbracket E(G) \rrbracket$.
\end{proof}

One might be tempted to short-circuit this step using the SCFA structure (i.e. units and counits) to provide $i_A$ and $o_B$, but this does not work. Consider the following example:
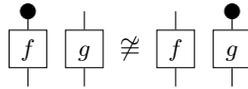
\[ \begin{tikzpicture}
	\begin{pgfonlayer}{nodelayer}
		\node [style=small box] (0) at (-0.75, 0) {$f$};
		\node [style=small box] (1) at (0.75, 0) {$g$};
		\node [style=black dot] (2) at (-0.75, 1) {};
		\node [style=none] (3) at (0.75, 1) {};
		\node [style=none] (4) at (-0.75, -1) {};
		\node [style=none] (5) at (0.75, -1) {};
	\end{pgfonlayer}
	\begin{pgfonlayer}{edgelayer}
		\draw (4.center) to (0);
		\draw (0) to (2);
		\draw (5.center) to (1);
		\draw (1) to (3.center);
	\end{pgfonlayer}
\end{tikzpicture} \ \not\cong \ \begin{tikzpicture}
	\begin{pgfonlayer}{nodelayer}
		\node [style=small box] (0) at (-0.75, 0) {$f$};
		\node [style=small box] (1) at (0.75, 0) {$g$};
		\node [style=none] (2) at (-0.75, 1) {};
		\node [style=black dot] (3) at (0.75, 1) {};
		\node [style=none] (4) at (-0.75, -1) {};
		\node [style=none] (5) at (0.75, -1) {};
	\end{pgfonlayer}
	\begin{pgfonlayer}{edgelayer}
		\draw (4.center) to (0);
		\draw (0) to (2.center);
		\draw (5.center) to (1);
		\draw (1) to (3);
	\end{pgfonlayer}
\end{tikzpicture} \]
If we attach dots everywhere:
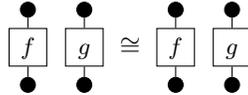
\[ \begin{tikzpicture}
	\begin{pgfonlayer}{nodelayer}
		\node [style=small box] (0) at (-0.75, 0) {$f$};
		\node [style=small box] (1) at (0.75, 0) {$g$};
		\node [style=black dot] (2) at (-0.75, 1) {};
		\node [style=black dot] (3) at (0.75, 1) {};
		\node [style=black dot] (4) at (-0.75, -1) {};
		\node [style=black dot] (5) at (0.75, -1) {};
	\end{pgfonlayer}
	\begin{pgfonlayer}{edgelayer}
		\draw (4) to (0);
		\draw (0) to (2);
		\draw (5) to (1);
		\draw (1) to (3);
	\end{pgfonlayer}
\end{tikzpicture} \ \cong \ \begin{tikzpicture}
	\begin{pgfonlayer}{nodelayer}
		\node [style=small box] (0) at (-0.75, 0) {$f$};
		\node [style=small box] (1) at (0.75, 0) {$g$};
		\node [style=black dot] (2) at (-0.75, 1) {};
		\node [style=black dot] (3) at (0.75, 1) {};
		\node [style=black dot] (4) at (-0.75, -1) {};
		\node [style=black dot] (5) at (0.75, -1) {};
	\end{pgfonlayer}
	\begin{pgfonlayer}{edgelayer}
		\draw (4) to (0);
		\draw (0) to (2);
		\draw (5) to (1);
		\draw (1) to (3);
	\end{pgfonlayer}
\end{tikzpicture} \]
the above proof would not be valid. Compare this to attaching new, freely-added boxes:
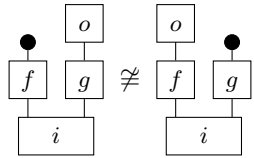
\[ \begin{tikzpicture}
	\begin{pgfonlayer}{nodelayer}
		\node [style=small box] (0) at (-0.75, 0) {$f$};
		\node [style=small box] (1) at (0.75, 0) {$g$};
		\node [style=black dot] (2) at (-0.75, 1) {};
		\node [style=none] (3) at (0.75, 1.25) {};
		\node [style=none] (4) at (-0.75, -1) {};
		\node [style=none] (5) at (0.75, -1) {};
		\node [style=medium box] (6) at (0, -1.5) {$i$};
		\node [style=small box] (7) at (0.75, 1.5) {$o$};
	\end{pgfonlayer}
	\begin{pgfonlayer}{edgelayer}
		\draw (4.center) to (0);
		\draw (0) to (2);
		\draw (5.center) to (1);
		\draw (1) to (3.center);
	\end{pgfonlayer}
\end{tikzpicture} \ \not\cong \ \begin{tikzpicture}
	\begin{pgfonlayer}{nodelayer}
		\node [style=small box] (0) at (-0.75, 0) {$f$};
		\node [style=small box] (1) at (0.75, 0) {$g$};
		\node [style=none] (2) at (-0.75, 1.25) {};
		\node [style=black dot] (3) at (0.75, 1) {};
		\node [style=none] (4) at (-0.75, -1) {};
		\node [style=none] (5) at (0.75, -1) {};
		\node [style=small box] (6) at (-0.75, 1.5) {$o$};
		\node [style=medium box] (7) at (0, -1.5) {$i$};
	\end{pgfonlayer}
	\begin{pgfonlayer}{edgelayer}
		\draw (4.center) to (0);
		\draw (0) to (2.center);
		\draw (5.center) to (1);
		\draw (1) to (3);
	\end{pgfonlayer}
\end{tikzpicture} \]

\section{Completeness of dagger-hypergraph categories}\label{sec:dagger-completeness}

We now turn to proving the completeness theorem of dagger-hypergraph categories for $\Mat(K)$, where $K$ is a field with non-trivial involution. Note that it is necessary take a non-trivial involution, otherwise new equations become true in all models. For example, scalars $s : I \to I$ would automatically satisfy $s^\dagger = s$ in all models, which is not provable by the dagger-hypergraph axioms. Also, the `tranposition' defined in terms of the Frobenius structure would become equal to the dagger:
\[ \left( \begin{tikzpicture}
	\begin{pgfonlayer}{nodelayer}
		\node [style=none] (0) at (0, 1) {};
		\node [style=none] (1) at (0, -1) {};
		\node [style=small box] (2) at (0, 0) {$f$};
	\end{pgfonlayer}
	\begin{pgfonlayer}{edgelayer}
		\draw (1.center) to (2);
		\draw (2) to (0.center);
	\end{pgfonlayer}
\end{tikzpicture} \right)^\dagger \ =\ \begin{tikzpicture}
	\begin{pgfonlayer}{nodelayer}
		\node [style=black dot] (0) at (-0.75, 1) {};
		\node [style=small box] (1) at (0, 0) {$f$};
		\node [style=none] (2) at (-1.5, 0.25) {};
		\node [style=none] (3) at (-1.5, -2) {};
		\node [style=black dot] (4) at (-0.75, 1.75) {};
		\node [style=none] (5) at (1.5, 2) {};
		\node [style=none] (6) at (1.5, -0.25) {};
		\node [style=black dot] (7) at (0.75, -1) {};
		\node [style=black dot] (8) at (0.75, -1.75) {};
	\end{pgfonlayer}
	\begin{pgfonlayer}{edgelayer}
		\draw [in=0, out=90, looseness=1.00] (1) to (0);
		\draw (3.center) to (2.center);
		\draw [in=180, out=90, looseness=1.00] (2.center) to (0);
		\draw (0) to (4);
		\draw (5.center) to (6.center);
		\draw [in=0, out=-90, looseness=1.00] (6.center) to (7);
		\draw (7) to (8);
		\draw [in=-90, out=180, looseness=1.25] (7) to (1);
	\end{pgfonlayer}
\end{tikzpicture} \]
which again is not provable by the dagger-hypergraph axioms. This is natural to consider these two expressions not to be equal because the LHS computes the \textit{conjugate-transpose} whereas the RHS computes the \textit{transpose}.

\begin{theorem}\label{thm:dagger-hypergraph-completeness}
  Two dot-diagrams $F$, and $G$ are isomorphic iff for all dagger-hypergraph functors $\llbracket-\rrbracket : \DotCat(\Sigma) \to \Mat(K)$, $\llbracket F \rrbracket = \llbracket G \rrbracket$.
\end{theorem}

\begin{proof}
  The proof is almost identical to that of Theorem~\ref{thm:hypergraph-completeness}. The main difference is that we define the functor $\llbracket - \rrbracket_F$ in terms of a different ring. Let \ZXX be the polynomial ring taking as variables $\mathbf X$ as before, along with a second copy $\overline{\mathbf X} := \{ \overline{X}_b \ |\ b \in B^F \}$. This becomes a ring with involution by taking the involution to be the ring homomorphism that interchanges $X_b \leftrightarrow \overline X_b$ and leaves the other ring elements fixed. We can then define
  \[ \llbracket - \rrbracket_F : \DotCat(\Sigma) \to \Mat(\ZXX) \]
  as:
  \[ \llbracket f \rrbracket_F :=
        \sum_{b \in (\Lbox^F)^{-1}(f)} f_b + 
        \sum_{b \in (\Lbox^F)^{-1}(f^\dagger)} (f_b)^\dagger \]
  Then, by a similar calculation to Theorem~\ref{thm:hypergraph-completeness} (and in fact, a nearly identical calculation to that in~\cite{SelingerCompleteness}) we get the following value for the closed, simple dot-diagram $G$:
  \begin{equation}\label{eq:big-sum}
    \llbracket G \rrbracket_F =
    \sum_{\phi \in \textrm{Idx}(G,\llbracket\rrbracket_F)}
    \sum_{\psi : B^G \to B^F}
    \prod_{b \in B^G}
    \begin{cases}
    X_{\psi(b)}
      & \textrm{if } \Lbox^F(\psi(b)) = \Lbox^G(b) \textrm{ and } \\
      & \phi(\Win^G(i,b)) = \Win^F(i,\psi(b)) \textrm{ and } \\
      & \phi(\Wout^G(b,j)) = \Wout^F(\psi(b),j) \textrm{ for all } i, j \\
    \overline X_{\psi(b)}
      & \textrm{if } \Lbox^F(\psi(b)) = \Lbox^G(b)^\dagger \textrm{ and } \\
      & \phi(\Wout^G(b,i)) = \Win^F(i,\psi(b)) \textrm{ and } \\
      & \phi(\Win^G(j,b)) = \Wout^F(\psi(b),j) \textrm{ for all } i, j \\
    0 & \textrm{otherwise}
    \end{cases}
  \end{equation}
  When $F$ and $G$ are simple and closed, we can still identify the `magic coefficient' of $\prod_b X_b$, whose value counts the cardinality of $\textrm{BBij}(G, F)$. Thus, if we first use the functor $\llbracket - \rrbracket^d$ to check if $F$ and $G$ have the same number of dots, then $\llbracket F \rrbracket_F = \llbracket G \rrbracket_F$ iff $F \cong G$.

  An involution-preserving semiring homomorphism will define a dagger-hypergraph functor from $\Mat(\ZXX)$ to $\Mat(K)$. Again, we define the functor $\textrm{ev} : \Mat(\ZXX) \to \Mat(K)$ in terms of an evaluation homomorphism, but this time one that sends $X_b$ to some element $k \in K$ which in turn fixes the value of $\overline{X}_b$ to be $\overline k$. Thus, the following is a dagger-hypergraph functor:
  \[ \llbracket - \rrbracket_{FG} := \textrm{ev}(\llbracket - \rrbracket_F) \]
  Note that we cannot choose \textit{all} the variables in $\llbracket F \rrbracket_{FG} - \llbracket G \rrbracket_{FG}$ independently, but rather only the $X_b$'s. However, whenever $K$ is characteristic zero and has a non-trivial involution, we can still choose values for $X_b$ such that $\llbracket F \rrbracket_{FG} - \llbracket G \rrbracket_{FG}$ is non-zero.\footnote{Thanks to Peter Selinger for pointing this out.} Thus, $\llbracket F \rrbracket_{FG} \neq \llbracket G \rrbracket_{FG}$.

  To extend from the simple, closed case, Corollaries~\ref{cor:non-simple} and~\ref{cor:non-closed} work unmodified.
\end{proof}


\bibliographystyle{akbib}
\bibliography{main}

\begin{thebibliography}{10}

\bibitem{AC1}
S.~Abramsky and B.~Coecke.
\newblock A categorical semantics of quantum protocols.
\newblock In {\em Proceedings of the 19th Annual IEEE Symposium on Logic in
  Computer Science (LICS)}, pages 415--425. IEEE Computer Society, 2004.
\newblock {a}rXiv:quant-ph/0402130.

\bibitem{CD2}
B.~Coecke and R.~Duncan.
\newblock Interacting quantum observables: categorical algebra and
  diagrammatics.
\newblock {\em New Journal of Physics}, 13:043016, 2011.
\newblock {arXiv:quant-ph/09064725}.

\bibitem{CPV}
B.~Coecke, D.~Pavlovi{\'c}, and J.~Vicary.
\newblock A new description of orthogonal bases.
\newblock {\em Mathematical Structures in Computer Science}, 2013.
\newblock {a}rXiv:quant-ph/0810.1037.

\bibitem{CPStar}
B.~Coecke, C.~Heunen, and A.~Kissinger.
\newblock Categories of quantum and classical channels.
\newblock arXiv:1305.3821 [quant-ph], 2013.

\bibitem{HasegawaHofmannPlotkin}
M.~Hasegawa, M.~Hofmann, and G.~D. Plotkin.
\newblock Finite dimensional vector spaces are complete for traced symmetric
  monoidal categories.
\newblock In A.~Avron, N.~Dershowitz, and A.~Rabinovich, editors, {\em Pillars
  of Computer Science}, volume 4800 of {\em Lecture Notes in Computer Science},
  pages 367--385. Springer, 2008.

\bibitem{AleksThesis}
A.~Kissinger.
\newblock Pictures of processes: Automated graph rewriting for monoidal
  categories and applications to quantum computing, 2012.
\newblock DPhil Thesis, Oxford University.

\bibitem{Lack}
S.~Lack.
\newblock Composing {PROP}s.
\newblock {\em Theory and Applications of Categories}, 13:147--163, 2004.

\bibitem{PaulAndreFunBox}
P.-A. Melli{\`e}s.
\newblock Functorial boxes in string diagrams.
\newblock In Z.~{\'E}sik, editor, {\em Computer Science Logic}, volume 4207 of
  {\em Lecture Notes in Computer Science}, pages 1--30. Springer Berlin
  Heidelberg, 2006.

\bibitem{SelingerCompleteness}
P.~Selinger.
\newblock Finite dimensional {H}ilbert spaces are complete for dagger compact
  closed categories (extended abstract).
\newblock {\em Electronic Notes in Theoretical Computer Science},
  270(1):113--119, 2011.

\bibitem{SelingerSelfDual}
P.~Selinger.
\newblock Autonomous categories in which {A} is isomorphic to {A}*.
\newblock In {\em Proceedings of the 7th International Workshop on Quantum
  Physics and Logic (QPL 2010)}, pages 151--160, 2010.

\end{thebibliography}

\end{document}